\documentclass[11pt,english]{article}
\usepackage[T1]{fontenc}
\usepackage[latin9]{inputenc}
\usepackage{geometry}
\usepackage{longtable}
\usepackage{amstext}
\usepackage{amssymb}
\usepackage{setspace}
\usepackage{esint}
\makeatletter

\newcommand{\binom}[2]{{#1 \choose #2}}

\providecommand{\tabularnewline}{\\}

\newtheorem{theorem}{Theorem}

\newtheorem{condition}{Condition}

\newtheorem{corollary}{Corollary}

\newtheorem{example}{Example}

\newtheorem{lemma}{Lemma}

\newtheorem{proposition}{Proposition}

\newenvironment{proof}[1][Proof]{\textbf{#1.} }{\ \rule{0.5em}{0.5em}}

\makeatother

\usepackage{babel}

\makeatother

\usepackage{babel}

\makeatother

\usepackage{babel}

\makeatother

\usepackage{babel}
\begin{document}

\title{Estimation for the Prediction of Point Processes with Many Covariates}

\author{Alessio Sancetta\thanks{Acknowledgement: I would like to thank the Co-Editor, the referees
and Luca Mucciante for comments that led to substantial improvements
both in content and presentation. E-mail: <asancetta@gmail.com>, URL:
<http://sites.google.com/site/wwwsancetta/>. Address for correspondence:
Department of Economics, Royal Holloway University of London, Egham
TW20 0EX, UK. }\\
}
\maketitle
\begin{abstract}
Estimation of the intensity of a point process is considered within
a nonparametric framework. The intensity measure is unknown and depends
on covariates, possibly many more than the observed number of jumps.
Only a single trajectory of the counting process is observed. Interest
lies in estimating the intensity conditional on the covariates. The
impact of the covariates is modelled by an additive model where each
component can be written as a linear combination of possibly unknown
functions. The focus is on prediction as opposed to variable screening.
Conditions are imposed on the coefficients of this linear combination
in order to control the estimation error. The rates of convergence
are optimal when the number of active covariates is large. As an application,
the intensity of the buy and sell trades of the New Zealand dollar
futures is estimated and a test for forecast evaluation is presented.
A simulation is included to provide some finite sample intuition on
the model and asymptotic properties. 


\end{abstract}

\section{Introduction}

Suppose that you want to estimate and then predict the likelihood
of a trade arrival for some financial instrument, which trades relatively
frequently. The reason for doing so could be market making or optimal
execution; such problems are quite common in the financial industry.
For example, in an application to be considered here, the instrument
is the futures on the New Zealand Dollar. A trade arrival for such
an instrument may depend on the state of the order book, which contains
5 levels on the bid and the offer. It may also depend on what happens
on other related instruments and past price and quoted volumes dynamics
as well as on past trades. The number of possible covariates can grow
quickly and become relatively large even for high frequency data. 

Problems such as the one just described can be addressed considering
trade arrivals as the jump of a counting process whose intensity (the
mean over an infinitesimal time period) depends on a set of covariates.
This paper considers the estimation of such counting processes for
problems where data are time series, the number of covariates is large,
and the functional form of the intensity does not need to be parametric. 

Let $\left(N\left(t\right)\right)_{t\geq0}$ be a counting process
with intensity measure 
\begin{equation}
\Lambda\left(A\right)=\int_{A}\exp\left\{ g_{0}\left(X\left(t\right)\right)\right\} dt,\label{EQ_intensityRepresentation}
\end{equation}
for any Borel set $A\subseteq[0,\infty)$, where $g_{0}$ is an unknown
function, $X\left(t\right)$ are $K$ dimensional covariates that
can depend on $t$. The intensity in (\ref{EQ_intensityRepresentation})
is understood to mean 
\[
\lim_{s\downarrow0}\frac{\Pr\left(N\left(t+s\right)-N\left(t\right)=1|\mathcal{F}_{t}\right)}{s}=\exp\left\{ g_{0}\left(X\left(t\right)\right)\right\} ,
\]
where $\mathcal{F}_{t}$ is the sigma algebra generated by $\left(N\left(s\right),X\left(s\right)\right)_{s\leq t}$.
Given that the covariates are time dependent, the intensity may depend
on the time elapsed from the last jump of $N\left(t\right)$. The
covariates are predictable, for example, adapted left continuous processes.
If conditioning on the covariates $X$, the process is Poisson, then
the counting process is usually referred to as Cox or doubly stochastic
process. 

Define the stopping times $T_{i}:=\inf\left\{ s>0:N\left(s\right)\geq i\right\} $,
$T_{0}=0$, i.e. $T_{i}$ is the the time of the $i^{th}$ jump. In
the empirical financial microstructure application to be considered
in this paper, the jump time $T_{i}$ is the time of the $i^{th}$
trade arrival for a specific security, and the covariates will be
information extracted from order book, amongst other quantities. The
statistical problem is the one where one observes $\left(N\left(t\right),X\left(t\right)\right)$
up to time $T$. By definition of the stopping times, waiting until
$T=T_{n}$ means that one observes $n$ jumps. The goal is to estimate
$g_{0}$. This function $g_{0}$ is only known to lie in some class
of additive functions, which will be introduced in due course. The
covariates and the durations between jumps are supposed to be stationary,
but neither independent nor Markovian. 

The time series problem where only one trajectory of the process is
observed and $g_{0}$ in (\ref{EQ_intensityRepresentation}) is possibly
nonlinear, and the number of covariates is large has not been previously
discussed in the literature. The framework allows us to deal with
ultra-high dimensional problems where the number of covariates is
exponentially larger than the the sample size ($n$ when $T=T_{n}$).
The covariates could be time series and lagged variables. This setup
is motivated by many applied problems such as the previously mentioned
trading arrival estimation problem (e.g., Bauwens and Hautsch, 2009,
for a survey and references for counting models applied to finance).
A traded instrument may depend on updates and information from other
instruments. This leads to a proliferation of the possible number
of variables even though one might expect that either only a handful
of them might be relevant, or many covariates could explain the intensity,
but with decreasing degree of importance. In the modelling application
in Section \ref{Section_EmpiricalApplication}, in one case, one ends
up with more than one thousand variables, with the number of trades
$n$ of about a thousand. 

The main technical features of the present study are: 1. estimation
of $g_{0}$ in (\ref{EQ_intensityRepresentation}), when $g_{0}$
is only known to lie in some large set of functions; 2. the number
of covariates is allowed to be larger than the number of observed
durations $n$; 3. a class of additive functions is defined, and it
is shown that within this class one can obtain convergence rates that
are optimal in the high dimensional case; 4. the estimation problem
can be solved by the Frank-Wolfe algorithm and rates of convergence
are given; 5. an empirical study provides applicability of the methodology
and a test for forecast superiority between counting models, showing
that suitably constrained large models can perform better out of sample. 

From a theoretical point of view, restrictions on the absolute summability
of linear coefficients (the $l_{1}$ norm of the coefficients) in
the additive model are imposed. Such Lasso kind of constraint tends
to produce models that are sparse. This means that if all the coefficients
are nonzero, but small, then tightening the constraint leads to many
variables that are zero and a few nonzero variables. On the other
hand, it is well known that tightening a constraint on the square
sum of the coefficients (i.e., $l_{2}$ norm as in ridge regression)
leads to all coefficients being small, but none being zero. 

From an empirical point of view, the paper considers estimation of
the intensity for the arrival of buy and sell trades on the New Zealand
dollar futures contract. The intensity is modelled using many covariates,
of the same order of magnitude as the number of durations. Estimation
of the intensity for buy and sell orders has been considered in the
literature (e.g., Hall and Hautsch 2007). However, no study seems
to consider market information (e.g., the order book) on the traded
instrument as well as other related instruments. The out of sample
results show that information provided by additional instruments is
relevant. To evaluate the out of sample performance of competing models,
an out of sample test based on the likelihood ratio is used. 

Details concerning the proofs and derivations in the main text are in the Supplementary Material. 

\subsection{Relation to the Literature}

In the regression context, high dimensional additive modelling has
been considered in the literature (e.g., Bühlmann and van de Geer,
2011, and references therein). This paper seems to be the first to
consider estimation with many covariates, allowing for a nonlinear
link function in a time series context. Here, time series means that
only one single realization of the process is observed over a window
expanding in the future. This framework differs from the one of Cox
proportional hazard model and Aalen multiplicative and additive model.
In that context, estimation with many variables has been considered
by various authors (e.g., Bradic et al., 2011, Gaiffas and Guilloux,
2012, amongst others). There, the focus is often in recovering the
true subset of active variables. This often results in stringent restrictions
on the covariates design and cross-dependence. Here, the focus is
on prediction and on weak conditions that can lead to consistency
even when the number of non-negligible covariates grows with the sample
size. Moreover, beyond additivity, the estimation considered here
is very general. Section \ref{Section_examplesApplications} provides
an overview of the applications. These include linear models, Hawkes
processes with covariates, threshold models, and additive monotone
functions amongst other possibilities. 

The analysis of estimators of the intensity function usually relies
on martingale methods (Andersen and Gill, 1982, van de Geer, 1995).
In the context of a fixed number of covariates, nonparametric estimators
are not uncommon (e.g., Nielsen and Linton, 1995, Fan et al., 1997,
are early references). 

The results derived here apply to parametric as well as to certain
nonparametric classes of functions. In the financial econometrics
literature, interest often lies in parametric modelling of a single
point process (e.g., Bauwens and Hautsch, 2009, for a survey). Hence,
the current paper considers the time series problem as in the financial
econometrics literature, but allowing for possibly nonparametric estimation
and for a large number of covariates as done in high dimensional statistics. 

In a time series context, the intensity is often modelled by Hawkes
processes. Loosely speaking, the intensity can be written as a predictable
function of durations (e.g., Bauwens and Hautsch, 2009). The framework
of this paper allows the aforementioned variables to be covariates. 

\subsection{Likelihood Estimation\label{Section_likelihood} }

It is well known (e.g., Brémaud, 1981, Ch.II, Theorem 16) that $\left\{ \Lambda\left((T_{i-1},T_{i}]\right):i\in\mathbb{N}\right\} $
($\Lambda$ as in (\ref{EQ_intensityRepresentation})) is i.i.d. exponentially
distributed with mean 1 when we condition on $\Lambda$. The likelihood
is easily derived from here assuming that $\Lambda$ has density $\lambda$
with respect to the Lebesgue measure (e.g., Ogata, 1978, eq.1.3).

Define the population log-likelihood 

\begin{equation}
L\left(g\right):=\mathbb{E}g\left(X\left(0\right)\right)\exp\left\{ g_{0}\left(X\left(0\right)\right)\right\} -\mathbb{E}\exp\left\{ g\left(X\left(0\right)\right)\right\} ,\label{EQ_populationLikelihood}
\end{equation}
assuming the expectations are well defined (see Section \ref{Section_proofPopulationLikelihood}
in the supplementary material). Suppose that $g_{0}$ in (\ref{EQ_intensityRepresentation})
lies in a set $\mathcal{G}$, momentarily assumed to be countable
to avoid distracting technicalities. Then, $g_{0}=\arg\sup_{g\in\mathcal{G}}L\left(g\right)$
using concavity of the log-likelihood. Given that expectations are
unknown, the above is replaced by the empirical estimator $g_{T}:=\arg\sup L{}_{T}\left(g\right)$,
where the $\sup$ is over some class of functions to be defined in
the next section and the sample likelihood is 
\begin{equation}
L_{T}\left(g\right):=\int_{0}^{T}g\left(X\left(t\right)\right)dN\left(t\right)-\int_{0}^{T}\exp\left\{ g\left(X\left(t\right)\right)\right\} dt,\label{EQ_negativeProfileLikelihood}
\end{equation}
where $L\left(g\right)=\lim_{T}L{}_{T}\left(g\right)/T$ almost surely
(see Section \ref{Section_proofPopulationLikelihood} in the supplementary
material for the proof of this statement). Supposing that one waits
until a time $T_{n}$ such that $N\left(T_{n}\right)=n$, the above
can be written as 
\[
L_{T_{n}}\left(g\right):=\sum_{i=1}^{n}\left[g\left(X\left(T_{i}\right)\right)-\int_{T_{i-1}}^{T_{i}}\exp\left\{ g\left(X\left(t\right)\right)\right\} dt\right].
\]
The representation in the last display is useful for actual computations,. 

\subsection{Outline of the Paper}

The plan for the paper is as follows. Section \ref{Section_modelg_0}
defines the model for the estimator and states the goal of the paper.
Section \ref{Section_MainResults} states the consistency result and
its optimality. A greedy algorithm is discussed as a method to carry
out the estimation in practice. Section \ref{Section_examplesApplications}
shows applications of the main result to a variety of estimation problems
and derives the convergence rates. Additional details are also given.
For example, an out of sample test based on the likelihood ratio is
suggested for forecast evaluation. Section \ref{Section_EmpiricalApplication}
applies the estimation procedure to the intensity of buy and sell
trades. Section \ref{Section_simulations} provides some finite sample
evidence to better understand the role of the different parameters
in the estimation. Section \ref{Section_Conclusion} contains some
further remarks. Proofs of the results are in Section \ref{Section_ProofsAppendix}
of the supplementary material. 

\section{The Model\label{Section_modelg_0}}

The goal is to allow for simple interpretation of the impact of the
covariates on the intensity. A good level of interpretability is gained
by letting $g\left(x\right)$ be linear in $x$. However, the impact
of each covariates might be nonlinear. For example, nonlinearities
are documented in many applications such as high frequency financial
data (e.g., Hasbrouck, 1991, Lillo et al., 2003). Whether, these nonlinearities
affect the intensity would depend on the application. An additive
nonlinear model is considered as a reasonable trade off between interpretability
and the possibility of nonlinear relations. In this case, $g\left(x\right)=\sum_{k=1}^{K}g^{\left(k\right)}\left(x\right)$,
where the $g^{\left(k\right)}$'s are some bounded functions, possibly
zero and for each $k$, $g^{\left(k\right)}\left(x\right)$ only depends
on $x_{k}$, the $k^{th}$ coordinate of $x=\left(x_{1},x_{2},...,x_{K}\right)$
(i.e., with abuse of notation, $g^{\left(k\right)}\left(x\right)=g^{\left(k\right)}\left(x_{k}\right)$).
This is done to reduce the notational burden. 

\subsection{Representation for Additive Functions\label{Section_functionRepresentation}}

For the purpose of controlling the estimation error, it is necessary
to impose some structure on the set within which the additive functions
are supposed to lie. Functions with the following structure are considered
\begin{equation}
g^{\left(k\right)}\left(x\right)=\sum_{\theta\in\Theta_{k}}b_{\theta}\theta\left(x\right)\label{EQ_gkRepresentation}
\end{equation}
where $\Theta_{k}$ is a set of functions that depends only on $x_{k}$,
$\Theta_{k}$ is a possibly uncountable set, and the $b_{\theta}$'s
are real valued coefficients. Given that $\Theta_{k}$ can be uncountable,
the above representation is more general than a standard series expansion.
The sum is understood to mean 
\[
\sum_{\theta\in\Theta_{k}}b_{\theta}:=\sup\left\{ \sum_{\theta\in F}b_{\theta}:\mathcal{H}\subseteq\Theta_{k},\,\mathcal{H}\,\text{is finite}\right\} .
\]
For example, we could have $g_{k}=b_{\theta}\theta$ for some $\theta\in\Theta_{k}$,
where $\Theta_{k}$ is a model, possibly infinite dimensional. In
consequence of the additive structure of $g$,
\begin{equation}
g\left(x\right)=\sum_{k=1}^{K}\left(\sum_{\theta\in\Theta_{k}}b_{\theta}\theta\left(x\right)\right),\label{EQ_g_additiveForm}
\end{equation}
where the terms in the parenthesis are just $g^{\left(k\right)}$
in (\ref{EQ_gkRepresentation}), which is a function that depends
on the $k^{th}$ covariate only. This structure is suitable for estimation.
Estimation within this framework requires choice of the $b_{\theta}$'s
as well as the $\theta$'s. For practical purposes the latter might
be simple parametric functions or fixed functions rather than general
infinite dimensional models. Details and examples are postponed to
Section \ref{Section_examplesApplications}. The interested reader
could skim through that section for an overview. In order to impose
general restrictions, suppose that the user fixes a set of weights
$\mathcal{W}:=\left\{ w_{\theta}\in\left(0,\infty\right):\theta\in\Theta\right\} $,
where $\Theta:=\bigcup_{k=1}^{K}\Theta_{k}$. This means that the
weights $w_{\theta}$ can be different for each function $\theta$
of the $k^{th}$ explanatory variable. Then, define $\mathcal{L}\left(B\right)=\mathcal{L}\left(B,\Theta,\mathcal{W}\right):=\left\{ \sum_{\theta\in\Theta}b_{\theta}\theta:\sum_{\theta\in\Theta}w_{\theta}\left|b_{\theta}\right|\leq B,w_{\theta}\in\mathcal{W}\right\} $.
This is a subset of the functions in (\ref{EQ_g_additiveForm}) such
that the weighted absolute sum of the coefficients is bounded by some
finite constant $B>0$. The weights are often used to control the
importance of each $\theta$. For example, one can let $w_{\theta}^{2}=Var\left(\theta\left(X\left(t\right)\right)\right)$
so that, intuitively, all functions have the same importance. A bound
on the weighted absolute sum of the regression coefficients is common
in Lasso estimation (e.g., Bühlmann and van de Geer, 2011). 

\begin{example}\label{Example_linearModel}Let $g\left(x\right)=\sum_{k=1}^{K}b_{k}X_{k}$
and $\pi_{k}$ be the map such that $\pi_{k}x=x_{k}$ for any $x\in\mathbb{R}^{K}$
and $x_{k}$ is the $k^{th}$ element in $x$. Then, $\Theta_{k}:=\left\{ \pi_{k}\right\} $
contains a single function that maps $x\in\mathbb{R}^{K}$ into its
$k^{th}$ co-ordinate $x_{k}$. Also, let $w_{\theta}^{2}=Var\left(X_{k}\left(0\right)\right)$
when $\theta\in\Theta_{k}$ and $X_{k}$ is the $k^{th}$ co-ordinate
of $X$. The constraint is $\sum_{k=1}^{K}\left|b_{k}\right|\sqrt{Var\left(X_{k}\left(0\right)\right)}\leq B$.
\end{example}

In other circumstances, the weight can serve the purpose of shrinkage
within each function $f_{k}$, which is important in infinite dimensional
spaces. 

\begin{example}To avoid distracting notation, suppose that $g\left(x\right)=\sum_{j=1}^{\infty}b_{j}x_{k}^{j}$,
a polynomial which depends on $x_{k}\in\left[0,1\right]$ only. Also
suppose that $\sum_{j=1}^{\infty}\left(j!\right)\left|b_{j}\right|\leq B$,
so that the weights force the coefficients to decay faster than $j!$.
An infinite differentiable function with derivatives of all orders
bounded by one can be written as the polynomial above where $\left|b_{j}\right|\leq\left(j!\right)^{-1}$.
Hence, the weights allow to account for this and the constraint induces
an additional shrinkage effect on the coefficients because of the
summability constraint. \end{example}

From now on, dependence on $\Theta$ and $\mathcal{W}$ will be implicit
when writing $\mathcal{L}\left(B\right)$. The approximation error
of functions in $\mathcal{L}\left(B\right)$ for some $B<\infty$
can be related to the bound on the absolute sum of the coefficients.
This is useful if one supposes that $g_{0}\in\mathcal{L}\left(B_{0}\right)$
for some unknown but finite $B_{0}$. If the user estimates the model
with $B<B_{0}$ an approximation error will be incurred. However,
note that the results of the paper will allow for more general forms
of misspecification. Let $P\left(x\right)$ be the marginal distribution
of $X\left(t\right)$, which by stationarity does not depend on $t$.
For any function $g:\mathbb{R}^{K}\rightarrow\mathbb{R}$, let $Pg=\int g\left(x\right)dP\left(x\right)$.
The $L_{r}\left(P\right)$ norm is $\left|\cdot\right|_{r}=\left(\int\left|\cdot\right|^{r}dP\right)^{1/r}$,
for $r\in[1,\infty)$, with the standard modification when $r=\infty$.
The following is a re-adaptation of a result in Sancetta (2015) and
can be used to control the approximation error of the estimator.

\begin{lemma}\label{Lemma_L_B_B0Approximation}Let $g_{0}\in\mathcal{L}\left(B_{0}\right)$
for some $B_{0}<\infty$ and $\bar{\theta}_{r}:=\sup_{\theta\in\Theta}\left|\theta\right|_{r}<\infty$.
Then, for any $B<\infty$, and $r\in\left[1,\infty\right]$, $\min_{g\in\mathcal{L}\left(B\right)}\left|g_{0}-g\right|_{r}\leq\underline{w}^{-1}\bar{\theta}_{r}\max\left\{ B_{0}-B,0\right\} $.\end{lemma}

When $g_{0}\notin\mathcal{L}\left(B\right)$, define the best uniform
approximation $g_{B}=\arg\inf\left|g-g_{0}\right|_{\infty}$, where
the infimum is over $\mathcal{L}\left(B\right)$. We shall define
\begin{equation}
B_{0}=\arg\inf_{B<\infty}\left|g_{B}-g_{0}\right|_{\infty}.\label{EQ_B0Definition}
\end{equation}
This means that $g_{B_{0}}$ is the best uniform approximation of
$g_{0}$ for any $g\in\bigcup_{B>0}\mathcal{L}\left(B\right)$. 

\subsection{The Goal}

The user supposes that $g_{0}\in\mathcal{L}\left(B_{0}\right)$, but
ignores the value of $B_{0}$. They guess a value $\bar{B}<\infty$.
If it is the case that $g_{0}\in\mathcal{L}\left(B_{0}\right),$ and
$\bar{B}\geq B_{0}$, there will be no approximation error. The estimation
error could be high, especially if $\bar{B}$ is much larger than
$B_{0}$. Once $\bar{B}$ is chosen, the log-likelihood in (\ref{EQ_negativeProfileLikelihood})
is maximized over $\mathcal{L}\left(\bar{B}\right)$. 

Let $\lambda=d\Lambda/d\mu$, where $\Lambda$ is the intensity measure
(\ref{EQ_intensityRepresentation}) and $\mu$ is the Lebesgue measure.
Then, $\lambda\left(X\left(t\right)\right)=\exp\left\{ g_{0}\left(X\left(t\right)\right)\right\} $
is as in (\ref{EQ_intensityRepresentation}). Suppose that $g$ is
fixed and bounded. Define the random norm 
\[
\left|g-g_{0}\right|_{\lambda,T}:=\sqrt{\frac{1}{T}\int_{0}^{T}\left(g\left(X\left(t\right)\right)-g_{0}\left(X\left(t\right)\right)\right)^{2}d\Lambda\left(t\right)}.
\]
By stationarity and ergodicity (e.g., Lemma 2 in Ogata, 1978), 
\begin{equation}
\left|g-g_{0}\right|_{\lambda,T}^{2}\rightarrow P\left(g-g_{0}\right)^{2}\lambda=\int\left(g\left(x\right)-g_{0}\left(x\right)\right)^{2}\lambda\left(x\right)dP\left(x\right),\label{EQ_empiricalPopulationNormConvergence}
\end{equation}
almost surely. The goal is to define an estimator $g_{T}$ in $\mathcal{L}\left(\bar{B}\right)$
and obtain rates of convergence to zero of $\left|g_{T}-g_{0}\right|_{\lambda,T}$.
By (\ref{EQ_empiricalPopulationNormConvergence}), this convergence
implies also convergence of $P\left(g_{T}-g_{0}\right)^{2}\lambda$,
though the rate of convergence for the latter cannot be derived unless
we impose dependence conditions on the covariates. If $\left|g_{0}\right|$
is bounded - as will be assumed here - the right hand side (r.h.s.)
of (\ref{EQ_empiricalPopulationNormConvergence}) is proportional
to $P\left(g-g_{0}\right)^{2}=\left|g-g_{0}\right|_{2}^{2}$, hence
the results to be derived also hold in $P$-integrated square error.
The proofs show that the convergence results hold for the Hellinger
distance between $\exp\left\{ g_{T}\right\} $ and $\exp\left\{ g_{0}\right\} $.
To minimise the notational burden, this is not explicitly stated in
the text. Details can be found in Section \ref{Section_ProofsAppendix}
of the supplementary material. Note that elements $g,g'\in\mathcal{L}\left(\bar{B}\right)$
will be considered the same if $P\left(g-g'\right)^{2}\lambda=0$. 

\subsubsection{Connection to Lasso }

Given the constraint on the coefficients $b_{\theta}$'s, minimization
over $\mathcal{L}\left(\bar{B}\right)$ is just the primal of an $l_{1}$
penalized likelihood estimator, i.e., Lasso. More specifically, conditioning
on the sample, for each $\bar{B}$, there is a constant $\pi_{\bar{B}}$
(the Lagrange multiplier, which increases with $T$, but at a possibly
different rate than $L_{T}$), such that the left hand side of the
following two displays are the same: 

\[
\arg\sup_{\theta,b_{\theta}}L{}_{T}\left(\sum_{\theta\in\Theta}b_{\theta}\theta\right),
\]
where the supremum is taken over those $\theta$'s and $b_{\theta}$'s
such that $\sum_{\theta\in\Theta}b_{\theta}\theta\in\mathcal{L}\left(\bar{B}\right)$;
\begin{equation}
\arg\sup_{\theta,b_{\theta}}L{}_{T}\left(\sum_{\theta\in\Theta}b_{\theta}\theta\right)-\pi_{\bar{B}}\sum_{\theta\in\Theta}w_{\theta}\left|b_{\theta}\right|,\label{EQ_lassoPenalised}
\end{equation}
where the supremum is taken over those $\theta$'s and $b_{\theta}$'s
such that $\theta\in\Theta$ and $b_{\theta}$ is a real number. If
$\mathcal{L}\left(\bar{B}\right)$ is a finite dimensional space,
$\pi_{\bar{B}}/T\rightarrow0$ (in probability) when the estimator
is consistent for $g_{0}$ inside $\mathcal{L}\left(\bar{B}\right)$.
However, when $\mathcal{L}\left(\bar{B}\right)$ is infinite dimensional,
norms are not equivalent and consistency under the norm we consider
in this paper does not mean consistency under the norm implied by
the constraint. Hence, for infinite dimensional $\mathcal{L}\left(\bar{B}\right)$,
$\pi_{\bar{B}}/T$ may converge to a constant even when the estimator
is consistent and $g_{0}$ lies inside $\mathcal{L}\left(\bar{B}\right)$. 

Estimation of the primal or dual problem gives the same solution when
we are able to map the constraint into the Lagrange multiplier $\pi_{\bar{B}}$.
In general, this is not straightforward. Solution for the Lasso problem
is often via co-ordinate descent, though rates of convergence are
usually not derived (e.g., Bühlmann and van de Geer, and references
therein). Here, we solve the constrained optimization and suggest
an algorithm to do so in practice and derive the convergence rates
of the algorithm (Section \ref{Section_EstimationDetails}). 

\section{Consistency of the Estimator\label{Section_MainResults}}

\subsection{Conditions}

The following conditions are imposed. Remarks on these are in Section
\ref{Section_remarksConditions}. To aid intuition, the conditions
can be divided into three groups: stochastic restrictions, parameter
space restrictions, and estimator restrictions. The conditions use
the notation defined around (\ref{EQ_intensityRepresentation}) and
in Section \ref{Section_functionRepresentation}.

\begin{condition}\label{Condition_stochasticRestrictions}Stochastic
Restrictions.
\begin{enumerate}
\item $\left(X\left(t\right)\right)_{t\geq0}$ is a stationary, ergodic,
predictable $K$ dimensional process with values in some set $\mathcal{X}\subseteq\mathbb{R}^{K}$
($K>1$);
\item The cumulative intensity $\Lambda$ has a density $\lambda$ with
respect to the Lebesgue measure (as in (\ref{EQ_intensityRepresentation})); 
\item $T_{0}=0$ is the time of the last jump before the jump at time $T_{1}$.
\end{enumerate}
\end{condition}

\begin{condition}\label{Condition_parameterSpaceRestrictions}Parameter
Space Restrictions.
\begin{enumerate}
\item The functions in $\Theta=\bigcup_{k=1}^{K}\Theta_{k}$ are measurable,
and uniformly bounded by a finite constant $\bar{\theta}:=\sup_{\theta\in\Theta}\sup_{x\in\mathcal{X}}\left|\theta\left(x\right)\right|$.
The set $\Theta_{k}$ has an $L_{\infty}\left(P\right)$ $\epsilon$-bracketing
number $N\left(\epsilon,\Theta_{k}\right)$ such that the entropy
integral $\int_{0}^{1}\sqrt{\ln\left(1+N\left(\epsilon,\Theta_{k}\right)\right)}d\epsilon$
is finite for every $k$ (not bounded and can grow with the sample
size); the weights in $\mathcal{L}\left(B,\Theta,\mathcal{W}\right)$
satisfy $\underline{w}:=\inf_{\theta\in\Theta}w_{\theta}>0$; 
\item In (\ref{EQ_intensityRepresentation}), $\bar{g}_{0}:=\left|g_{0}\right|_{\infty}<\infty$
and if $g_{0}\neq g_{B_{0}}$, then $B_{0}<\infty$ (see (\ref{EQ_B0Definition})).
\end{enumerate}
\end{condition}

\begin{condition}\label{Condition_estimatorRestrictions}Estimator
Restrictions.

The estimator $g_{T}$ satisfies:
\begin{enumerate}
\item $\Pr\left(g_{T}\notin\mathcal{L}\left(\bar{B},\Theta,\mathcal{W}\right)\right)=o\left(1\right)$; 
\item $L_{T}\left(g_{T}\right)\geq\sup_{g\in\mathcal{L}\left(\bar{B},\Theta,\mathcal{W}\right)}L{}_{T}\left(g\right)-O_{p}\left(\frac{T}{r_{T}^{2}}\right)$,
where $r_{T}$ is as in (\ref{EQ_r_T}) in Section \ref{Section_Consistency}. 
\end{enumerate}
\end{condition}

In general, from (\ref{EQ_r_T}) one can deduce that $r_{T}^{2}\lesssim T^{1/2}$,
where, throughout, $\lesssim$ is inequality up to a multiplicative
universal finite absolute constant.

\subsection{Consistency Results\label{Section_Consistency}}

It will be shown that the overall complexity in the statistical estimation
depends on three factors. The logarithm of the number of variables
$K$, $\bar{B}$ (in $\mathcal{L}\left(\bar{B}\right)$) and the entropy
integral of the largest of the sets $\Theta_{k}$ . To ease notation,
dependence on $\bar{\theta}$ and $\underline{w}$ is suppressed in
what follows. More explicit bounds can be found in the proof of the
results. 

\begin{theorem}\label{Theorem_consistency}Suppose that there is
a nondecreasing sequence $r_{T}$ such that
\begin{equation}
r_{T}^{2}\lesssim\min\left\{ \frac{\bar{B}^{-1}T^{1/2}}{\sqrt{\ln K}+\max_{k\leq K}\int_{0}^{1}\sqrt{\ln\left(1+N\left(\epsilon,\Theta_{k}\right)\right)}d\epsilon},\frac{1}{\inf_{g\in\mathcal{L}\left(\bar{B}\right)}\left|g-g_{0}\right|_{\infty}^{2}}\right\} .\label{EQ_r_T}
\end{equation}
Under Conditions \ref{Condition_stochasticRestrictions}, \ref{Condition_parameterSpaceRestrictions},
and \ref{Condition_estimatorRestrictions}, $\left|g_{T}-g_{0}\right|_{\lambda,T}=O_{p}\left(r_{T}^{-1}\right)$.
\end{theorem}

Note that the condition that $r_{T}$ is nondecreasing implicitly
imposes restrictions on $\bar{B}$, $K$ and $N\left(\epsilon,\Theta_{k}\right)$.
The daunting expression (\ref{EQ_r_T}) does simplify, but it is stated
in this form for flexibility. Section \ref{Section_examplesApplications}
considers applications of this result to a variety of problems so
that the bound becomes considerably simple. To provide a sense for
the sharpness of the bound, it might be convenient to suppose that
the approximation error $\inf_{g\in\mathcal{L}\left(\bar{B}\right)}\left|g-g_{0}\right|_{\infty}$
is zero. Also suppose that the entropy integral is bounded by a finite
constant. In this case, the rate of convergence of $\left|g_{T}-g_{0}\right|_{\lambda,T}$
is $O\left(\left(\ln\left(K\right)/T\right)^{1/4}\right)$. By stationarity
and ergodicity, it is easy to see that for $T=T_{n}$ ($T_{n}$ is
the time of the $n^{th}$ jump), $T_{n}\asymp n$, where $\asymp$
means equality up to a multiplicative finite absolute constant. In
consequence, the bound becomes the more familiar $O\left(\left(\ln\left(K\right)/n\right)^{1/4}\right)$
for $K>1$. Results in Tsybakov (2003) show that in a regression context
with Gaussian errors, no linear estimator of the convex combination
of $K$ bounded terms can achieve a rate faster than $O_{p}\left(\left(\ln\left(K\right)/n\right)^{1/4}\right)$
when $K$ is of larger order of magnitude than $n^{1/2}$ (see Theorem
2 in Tsybakov, 2003). Hence, without further assumptions, one can
suppose that the convergence rate derived here is optimal in this
context. Theorem \ref{Theorem_optimality} in Section \ref{Section_optimality}
lends some rigor to this supposition. 

Now to show the effect of the approximation error when $g_{0}\in\mathcal{L}\left(B_{0}\right)$
for some unknown, but finite $B_{0}$, consider the following scenario.
Let $\bar{B}\rightarrow\infty$ so that eventually $\bar{B}\geq B_{0}$.
By Lemma \ref{Lemma_L_B_B0Approximation} deduce that the approximation
error is eventually exactly zero for some finite $\bar{B}$. In consequence
the following holds.

\begin{corollary}Suppose that $g_{0}\in\mathcal{L}\left(B_{0}\right)$.
Under the conditions of Theorem \ref{Theorem_consistency}, for any
$\bar{B}\rightarrow\infty$, 
\[
\left|g_{T}-g_{0}\right|_{\lambda,T}^{2}=O_{p}\left(\frac{\bar{B}\left[\sqrt{\ln K}+\max_{k\leq K}\int_{0}^{1}\sqrt{\ln\left(1+N\left(\epsilon,\Theta_{k}\right)\right)}d\epsilon\right]}{T^{1/2}}\right)
\]
 \end{corollary} 

More generally, when $g_{0}\notin\mathcal{L}\left(B\right)$ for any
$B$, the approximation error in Theorem \ref{Theorem_consistency}
can be bounded using the following, which follows from the triangle
inequality and Lemma \ref{Lemma_L_B_B0Approximation}. 

\begin{lemma}\label{Lemma_LInfinityApproximation}Under Condition
\ref{Condition_parameterSpaceRestrictions}, 
\[
\inf_{g\in\mathcal{L}\left(\bar{B}\right)}\left|g_{0}-g\right|_{\infty}\lesssim\inf_{B>0}\left\{ \max\left\{ B-\bar{B},0\right\} +\inf_{g\in\mathcal{L}\left(B\right)}\left|g_{0}-g\right|_{\infty}\right\} .
\]
\end{lemma}

The reader interested in the scope of the possible applications can
go directly to Section \ref{Section_examplesApplications}. The next
Sections provide remarks on optimality, the conditions, and details
on the solution of the estimation problem. 

\subsection{Optimality\label{Section_optimality}}

From the previous remarks, it is reasonable to infer that the rates
of convergence in Theorem \ref{Theorem_consistency} are optimal for
$K$ large. To avoid technicalities, consider the following simplified
scenario. One may argue that less stringent conditions should make
the estimation problem harder and in consequence, if the lower bound
holds under restrictive conditions, it should hold under more general
ones. Recall that $X_{k}\left(t\right)$ is the $k^{th}$ element
in the vector of covariates $X\left(t\right)$. 

\begin{theorem}\label{Theorem_optimality}Suppose $\mathcal{L}\left(1\right):=\mathcal{L}\left(1,\Theta,\mathcal{W}\right)$,
where the sets $\Theta_{k}$'s contain bounded functions, and the
weights in $\mathcal{W}$ have been set to one. Suppose that $X\left(t+T_{i-1}\right)=X\left(T_{i-1}\right)$
for $t\in(0,T_{i}]$, i.e., $X\left(t\right)$ is constant between
jumps of the point process $N$, and $T_{0}=0$. Also suppose that
$\left(X\left(T_{i}\right)\right)_{i\geq0}$ forms a sequence of i.i.d.
random variables and that the $X_{k}\left(T_{i}\right)$'s are independent
across $k$, with continuous distribution function. For $K>T_{n}^{1/2}$
with $K=O\left(T_{n}^{p}\right)$ for any $p<\infty$, and $n\rightarrow\infty$,
\[
\inf_{g_{T}}\sup_{g_{0}\in\mathcal{L}\left(1\right)}\int_{0}^{T_{n}}\left|g_{T_{n}}\left(X\left(t\right)\right)-g_{0}\left(X\left(t\right)\right)\right|^{2}\exp\left\{ g_{0}\left(X\left(t\right)\right)\right\} dt\gtrsim\sqrt{T_{n}\ln\left(1+KT_{n}^{-1/2}\right)}
\]
in probability, where the infimum is taken over all possible estimators
$g_{T_{n}}$ of the intensity. \end{theorem}

Theorem \ref{Theorem_optimality} says that even under rather restrictive
conditions, as long as the number of variables $K$ is of order of
magnitude greater than $T_{n}^{1/2}$, the convergence rate under
$\left|\cdot\right|_{\lambda}$ cannot be faster than $\sqrt{T_{n}^{-1/2}\ln K}$. 

\subsection{Remarks on Conditions\label{Section_remarksConditions}}

It is worth emphasizing that the conditions do not restrict $g_{0}\in\mathcal{L}\left(B\right)$
for some $B<\infty$.

Condition \ref{Condition_stochasticRestrictions} is mild. For all
practical cases, one usually restricts $X$ to be an adapted process
that is left continuous. This implies predictability (e.g., Brémaud,
1981). In consequence, the time from last jump $R\left(t\right):=\inf\left\{ t-T_{i}:t-T_{i}>0,i=0,1,2...\right\} $
can be used as a covariate, as it is a predictable process. This will
be the case when estimating certain nonlinear Hawkes processes in
Section \ref{Section_examplesApplications}. The fact that $T_{0}=0$
is used to keep notation simple. Similarly, the condition $K>1$ is
used to avoid writing $\ln\left(1+K\right)$ instead of $\ln K$ in
various places. 

In Condition \ref{Condition_parameterSpaceRestrictions}, the entropy
integral restriction on the class of functions is standard. It is
needed as the framework is quite general, hence it requires some control
of the complexity of the functions in $\Theta_{k}$. The entropy integral
is finite, but can grow with the sample size, even though this is
not made explicit in the notation (see Section \ref{Section_MonotoneBernstein},
and the proof of Lemma \ref{Lemma_hawkesApproximation} in the supplementary
material). The $L_{\infty}\left(P\right)$ $\epsilon$-bracketing
number of a set $\Theta_{k}$ is the number of pairs of elements in
some set $\mathcal{V}$ such that for each $\theta\in\Theta_{k}$,
there is a bracket $\left[\theta_{L},\theta_{U}\right]$ satisfying
$\theta_{L}\leq\theta\leq\theta_{U}$, and $\left|\theta_{L}-\theta_{U}\right|_{\infty}\leq\epsilon$.
The uniform norm can be replaced by the random norm $T^{-1}\int_{0}^{T}\left|\theta_{L}-\theta_{U}\right|^{2}d\Lambda$,
which is actually the norm used in the proofs. This is difficult to
control and in the applications considered in this paper, the (stronger)
uniform norm is used instead. To cover the case of sieve estimation
and/or misspecification, $g_{0}$ is not restricted to lie in $\mathcal{L}$,
but needs to be uniformly bounded. 

Condition \ref{Condition_estimatorRestrictions} only requires that,
asymptotically, the estimators satisfy the complexity restrictions
discussed in the present paper. This is weaker than assuming that
the absolute sum of the coefficients is bounded by $\bar{B}$ for
any sample size and that the estimators of the functions $\theta_{k}$'s
are always in $\Theta_{k}$. This setup allows to cover different
approaches for estimation without restricting attention to a specific
one. Moreover, the estimator $g_{T}$ only needs to maximize the sample
likelihood $L_{T}$ asymptotically, rather than exactly. Section \ref{Section_EstimationDetails}
provides details on a computationally feasible estimation method. 

In some circumstances, we do not observe the true covariates, and
can only estimate the intensity using approximate data, which may
not even be stationary. A typical example is in the context of Hawkes
processes (see Section \ref{Section_examplesApplications}) or when
a covariate is for example a moving average of the past values. In
the aforementioned cases, the true covariates are a causal filter
of some quantity, but we can only construct the filter using some
initial condition rather than observations prior to time $T_{0}=0$.
Note that the true covariates still satisfy Condition \ref{Condition_stochasticRestrictions}.
However, we carry out optimization on some surrogate data so that
the last point in Condition \ref{Condition_estimatorRestrictions}
does not directly hold. The following allows us to consider such cases.

\begin{corollary}\label{Lemma_asymptoticXs}Suppose Conditions \ref{Condition_stochasticRestrictions}
and \ref{Condition_parameterSpaceRestrictions} hold and let $r_{T}$
be as in (\ref{Theorem_consistency}). Define $\bar{B}_{w}:=\bar{B}/\underline{w}$.
Let $\tilde{X}\left(t\right)$ be arbitrary covariates, but such that
\begin{equation}
\mathbb{E}\sup_{\theta\in\Theta}\int_{0}^{T}\left|\theta\left(\tilde{X}\left(t\right)\right)-\theta\left(X\left(t\right)\right)\right|dt=O\left(e^{-\bar{B}_{w}\bar{\theta}}\sqrt{T\ln K}\right).\label{eq_stationaryApproximation}
\end{equation}
Suppose that $\tilde{g}_{T}$ satisfies $\Pr\left(\tilde{g}_{T}\notin\mathcal{L}\left(\bar{B},\Theta,\mathcal{W}\right)\right)=o\left(1\right)$,
and
\begin{equation}
\tilde{L}{}_{T}\left(\tilde{g}_{T}\right)\geq\sup_{g\in\mathcal{L}\left(\bar{B},\Theta,\mathcal{W}\right)}\tilde{L}{}_{T}\left(g\right)-O_{p}\left(\frac{T}{r_{T}^{2}}\right)\label{EQ_surrogateLogLikApproxMax}
\end{equation}
where $\tilde{L}{}_{T}$ is the loglikelihood $L_{T}$ when we use
covariates $\tilde{X}\left(t\right)$ instead of $X\left(t\right)$,
as data. Then, $\tilde{g}_{T}$ is also an approximate minimizer of
$L_{T}$, i.e., it satisfies Condition \ref{Condition_estimatorRestrictions}
(with error $O_{p}\left(T/r_{T}^{2}\right)$). Hence, 
\[
\left|\tilde{g}_{T}-g_{0}\right|_{\lambda}^{2}=\int_{0}^{T}\left|\tilde{g}_{T}\left(X\left(t\right)\right)-g_{0}\left(X\left(t\right)\right)\right|^{2}\exp\left\{ g_{0}\left(X\left(t\right)\right)\right\} =O_{p}\left(r_{T}^{-2}\right).
\]
Moreover, 
\[
\int_{0}^{T}\left|\tilde{g}_{T}\left(\tilde{X}\left(t\right)\right)-g_{0}\left(X\left(t\right)\right)\right|^{2}\exp\left\{ g_{0}\left(X\left(t\right)\right)\right\} dt=O_{p}\left(r_{T}^{-2}\right)
\]
with $r_{T}^{2}$ as in (\ref{Theorem_consistency}).\end{corollary}

Corollary \ref{Lemma_asymptoticXs} says that we obtain the same rates
of convergence even when the estimator is computed from the loglikelihood
$\tilde{L}_{T}$ based on surrogate covariates, as long as the surrogate
covariates satisfy (\ref{eq_stationaryApproximation}). The last display
in Corollary \ref{Lemma_asymptoticXs} says that $\tilde{g}_{T}\left(\tilde{X}\left(t\right)\right)$
is close to $g_{0}\left(X\left(t\right)\right)$ even though they
are evaluated at different data. 

\subsection{Estimation Algorithm \label{Section_EstimationDetails}}

Maximization of the log-likelihood over $\mathcal{L}\left(\bar{B}\right)$
leads to a unique maximum (within an equivalence class) because of
concavity of the objective function and the convex and closed constraint.
However, while suitable for theoretical derivations, it is too abstract
for practical implementation. The algorithm in Figure 1 can be used
to solve the constrained minimization. For real valued functions $g$
and $h$ on $\mathbb{R}^{K}$, the following derivative of the log-likelihood
in the direction of a function $h$ is used
\[
D_{T}\left(g,h\right):=\int_{0}^{T}h\left(X\left(t\right)\right)dN\left(t\right)-\int_{0}^{T}h\left(X\left(t\right)\right)\exp\left\{ g\left(X\left(t\right)\right)\right\} dt.
\]
There is a line search to find the coefficient $\rho_{j}$. To speed
up the computations, this can be set to the deterministic value $\rho_{j}=2/\left(j+1\right)$.
The updated approximation to the constrained maximum at step $j$
is denoted by $F_{j}$. The bound to be given in Theorem \ref{Theorem_FWA}
holds in this case as well.\\

\begin{tabular}{l}
Figure 1. Log-Likelihood Optimization\tabularnewline
\hline 
Set:\tabularnewline
$m\in\mathbb{N}$\tabularnewline
$F_{0}:=0$\tabularnewline
$\bar{B}<\infty$\tabularnewline
For: $j=1,2,...,m$\tabularnewline
$\theta_{j}:=\arg\sup_{\theta\in\Theta}\left|D_{T}\left(F_{j-1},\theta\right)\right|/w_{\theta}$\tabularnewline
$b_{j}:=\frac{\bar{B}}{w_{\theta}}sign\left(D_{T}\left(F_{j-1},\theta_{j}\right)\right)$\tabularnewline
$\rho_{j}:=\arg\max_{\rho\in\left[0,1\right]}L{}_{T}\left(\left(1-\rho\right)F_{j-1}+\rho b_{j}\theta_{j}\right)$
or $\rho_{j}:=2/\left(j+1\right)$ \tabularnewline
$F_{j}\left(X\right):=\left(1-\rho_{j}\right)F_{j-1}\left(X\right)+\rho_{j}b_{j}\theta_{j}\left(X\right)$\tabularnewline
\tabularnewline
\end{tabular}

\begin{theorem}\label{Theorem_FWA}Let $F_{m}$ be the resulting
estimator from Figure 1. Define $\bar{B}_{w}:=\bar{B}/\underline{w}$.
Then, 
\[
L_{T}\left(F_{m}\right)\geq\sup_{g\in\mathcal{L}\left(\bar{B}\right)}L{}_{T}\left(g\right)-\frac{8Te^{\bar{B}_{w}\bar{\theta}}\left(\bar{B}_{w}\bar{\theta}\right)^{2}}{m+2}
\]
where the notation is from Condition \ref{Condition_parameterSpaceRestrictions}.

\end{theorem}

The algorithm in Figure 1 belongs to the family of Frank-Wolfe algorithms
(e.g., Jaggi, 2013, for the general proof of the convergence towards
the optimum point, and Sancetta, 2016, for its statistical properties
for linear models). The following identifies a suitable number of
iterations for the purpose of consistent estimation. 

\begin{corollary}If $m^{-1}=o\left(T^{-1/2}e^{-\bar{B}_{w}\bar{\theta}}\left(\bar{B}_{w}\bar{\theta}\right)^{-2}\right)$,
then $F_{m}$ in Figure 1 satisfies Condition \ref{Condition_estimatorRestrictions}.
Hence, if $\bar{B}$ is bounded, $m^{-1}=o\left(T^{-1/2}\right)$.
\end{corollary}

\subsection{Application to Various Estimation Methods and Model Specifications\label{Section_examplesApplications}}

The class of functions is general and can accommodate various estimation
methods and model specifications. Below, different models, function
classes and estimators are discussed. There is some overlap for some
of the applications, but the variations in terms of approximation
error make them different enough to justify their individual treatment. 

To avoid some oddities in the discussion, define the map $\left(x_{1},x_{2},...,x_{K}\right)=x\mapsto\pi_{k}\left(x\right)=x_{k}$
so that by composition, for any $f$ on $\mathbb{R}$, $f\circ\pi_{k}\left(x\right)=f\left(x_{k}\right)$.
In all the examples, it is tacitly assumed that the support of each
covariate is $\left[0,1\right]$. This is done for simplicity to avoid
distracting technicalities even when not necessary. In various occasions,
we may have a nontrivial approximations error. In this case, the following
will be used to indicate a set that contains the true $g_{0}$, 
\begin{equation}
\mathcal{G}\left(B\right):=\left\{ g=\sum_{k=1}^{K}b_{k}f_{k}\circ\pi_{k}:f_{k}\in\mathcal{H},\,\sum_{k=1}^{K}\left|b_{k}\right|\leq B\right\} ,\label{EQ_GTrueSet}
\end{equation}
where $\mathcal{H}$ is a class of univariate functions which will
be defined within each section below, depending on the application.
In all the examples of this section, all the weights $w_{\theta}$'s
in $\mathcal{W}$ are supposed to be equal to one without further
mention. Then, when $\Theta_{k}=\left\{ f\circ\pi_{k}:f\in\mathcal{H}\right\} $,
$\mathcal{L}\left(B\right)=\mathcal{G}\left(B\right)$. Suppose that
$f_{V,k}$ is an approximation to a function $f_{k}\in\mathcal{H}$,
then 
\begin{equation}
\left|\sum_{k=1}^{K}b_{k}f_{k}-\sum_{k=1}^{K}b_{k}f_{V,k}\right|_{\infty}\leq B\max_{k\leq K}\left|f_{k}-f_{V,k}\right|_{\infty}\label{EQ_approximationByUnivariate}
\end{equation}
when $\sum_{k=1}^{K}\left|b_{k}\right|\leq B$. This will be used
in some of the examples, in order to estimate the approximation error.
In this case, (\ref{EQ_approximationByUnivariate}) will be used in
conjunction with Lemma \ref{Lemma_LInfinityApproximation} where $B$
is just a bounded constant (e.g., $B=B_{0}$). Finally, to avoid trivialities
$K>1$ in all the bounds below. The bounds are of particular interest
when $K\gtrsim T^{1/2}$. Note that in the examples, we can have bounds
such as $\left|g_{T}-g_{0}\right|_{\lambda,T}^{2}\lesssim\bar{B}\sqrt{\left(\ln K\right)/T}$.
It is tacitly assumed that we require the r.h.s. to be $O\left(1\right)$.
Proofs of the following corollaries to Theorem \ref{Theorem_consistency}
can be found in Section \ref{Section_proofOfLemma1Corollaries} of
the supplementary material.

\subsubsection{Linear Model with Many Variables\label{Section_ExampleLinear}}

Let $\Theta_{k}:=\left\{ \pi_{k}\right\} $ which maps $x\in\mathbb{R}^{K}$
into its $k^{th}$ co-ordinate $x_{k}$. Then, $g\left(x\right)=\sum_{k=1}^{K}b_{k}x_{k}$.
The following holds.

\begin{corollary}\label{Corollary_linear}Suppose that $g_{0}\in\mathcal{L}\left(\bar{B}\right)$.
Under Conditions \ref{Condition_stochasticRestrictions} and \ref{Condition_estimatorRestrictions},
$\left|g_{T}-g_{0}\right|_{\lambda,T}^{2}\lesssim\left(\frac{\ln K}{T}\right)^{1/2}$
in probability.\end{corollary}

The corollary implies that the estimator is consistent even in the
ultra high dimensional case $K=O\left(e^{T^{c}}\right)$ for $c\in[0,1)$. 

\subsubsection{Hawkes Process with Many Covariates\label{Section_hawkesProcess}}

There are many version of the Hawkes process. For the sake of illustration,
consider a nonlinear function of the standard exponential decay case
(e.g., Brémaud and Massoulié, 1996). Define the family of processes
$\left\{ \left(\tilde{f}_{a}\left(t\right)\right)_{t\geq0}:a\in\left[\underline{a},\bar{a}\right]\subset\left(0,\infty\right)\right\} $,
where for each $a$, $\tilde{f}_{a}\left(t\right):=f\left(\int_{[0,t)}e^{-a\left(t-s\right)}dN\left(s\right)\right)$
and $f$ is a bounded Lipschitz function. The process $\tilde{f}_{a}\left(t\right)$
is not stationary because it is initiated at $t=0$. In consequence,
by Condition \ref{Condition_stochasticRestrictions}, it cannot be
used as one of the covariates. Define the family $\left\{ \left(f_{a}\left(t\right)\right)_{t\geq0}:a\in\left[\underline{a},\bar{a}\right]\subset\left(0,\infty\right)\right\} $,
where $f_{a}\left(t\right)=f\left(\int_{\left(-\infty,t\right)}e^{-a\left(t-s\right)}dN\left(s\right)\right)$
and $f$ is as before. The processes $f_{a}$'s are stationary, but
not observable. Despite the notational difference, one can verify
the condition of Corollary \ref{Lemma_asymptoticXs} to see that Theorem
\ref{Theorem_consistency} still holds. We also need to verify that
using $f_{a}\left(t\right)$ the counting process is stationary.

\begin{corollary}\label{Lemma_hawkesApproximation}Under Condition
\ref{Condition_stochasticRestrictions}, the point process with intensity
density $\lambda\left(t\right)=\exp\left\{ f_{a_{0}}\left(t\right)+g_{0}\left(X\left(t\right)\right)\right\} $
(for any $a_{0}\in\left(\underline{a},\bar{a}\right)$) has a stationary
distribution. Moreover, suppose that the loglikelihood with intensity
$\exp\left\{ \tilde{f}_{a}\left(t\right)+g\left(X\left(t\right)\right)\right\} $
is maximized w.r.t. $g\in\mathcal{L}\left(\bar{B}\right)$ and $a\in\left[\underline{a},\bar{a}\right]$
by $g_{T}$ and $a_{T}$ (even approximately with same error as in
Condition \ref{Condition_stochasticRestrictions}). Suppose that $\bar{B}$
is fixed, and $g_{0}\in\mathcal{L}\left(\bar{B}\right)$, then, in
probability, 
\begin{equation}
\left|\left(g_{T}+\tilde{f}_{T}\right)-\left(g_{0}+f_{0}\right)\right|_{\lambda,T}^{2}\lesssim\frac{\sqrt{\ln K}+\sqrt{\ln T}+\max_{k\leq K}\int_{0}^{1}\sqrt{\ln\left(1+N\left(\epsilon,\Theta_{k}\right)\right)}d\epsilon}{\sqrt{T}}.\label{EQ_hawkesBoundRate}
\end{equation}
Also suppose that $\Theta_{k}:=\left\{ \pi_{k}\right\} $, then $\left|\left(g_{T}+\tilde{f}_{T}\right)-\left(g_{0}+f_{0}\right)\right|_{\lambda,T}^{2}\lesssim\left(\frac{\ln KT}{T}\right)^{1/2}$
in probability. \end{corollary}

Note that to ease notation, we use $\ln KT=\ln\left(KT\right)$ and
similarly, throughout.

\subsubsection{Threshold Model with Many Variables\label{Section_thresholdModel}}

Suppose that $\varphi:\mathbb{R}\rightarrow\left[0,1\right]$ is Holder's
continuous with parameter $\alpha\in(0,1]$, i.e., $\left|\varphi\left(x\right)-\varphi\left(y\right)\right|\lesssim\left|x-y\right|^{\alpha}$.
Consider the class of linear threshold functions $f\left(x,z\right):=a_{1}x+a_{2}x\varphi\left(c_{1}z-c_{2}\right)$,
$x,z\in\mathbb{R}$, where $a_{1},a_{2},c_{1},c_{2}$ are unknown
real coefficients, with $a_{1},a_{2},c_{1},c_{2}\in\left[-1,1\right]$.
Denote the set of such functions by $\mathcal{H}$. 

Let $\left(Z\left(t\right)\right)_{t\geq0}$ be a predictable stationary
and ergodic real valued process taking values in $\left[0,1\right]$
as for the $X_{k}$'s. refer to it as a threshold variable. Then,
$f\left(X_{k}\left(t\right),Z\left(t\right)\right)$ is a transition
process, for the $k^{th}$ covariate: the impact of $X_{k}$ depends
on the threshold variable $Z$. Hence, $f\left(x,z\right)$ is a smooth
transition function (see van Dijk et al., 2002, for a survey of smooth
regression models based on this functional specification). 

The class of functions with elements $\varphi\left(c_{1}z-c_{2}\right)$
with bounded $z$ has finite entropy integral (e.g., deduce this from
Theorem 2.7.11 in van der Vaart and Wellner, 2000). Given that $a_{1},a_{2}\in\left[-1,1\right]$,
it follows that $\mathcal{H}$ has finite entropy integral. Let $\Theta_{k}:=\left\{ f\circ\left(\pi_{k},\iota\right):f\in\mathcal{H}\right\} $,
where $\iota$ is the identity map $\iota\left(z\right)=z$ (i.e.,
$f\circ\left(\pi_{k}x,\iota z\right)=f\left(x_{k},z\right)$).

\begin{corollary}\label{Corollary_thresholdModel}Let $Z$ be as
described before. Suppose that $g_{0}\in\mathcal{L}\left(B_{0}\right)$.
Under Conditions \ref{Condition_stochasticRestrictions} and \ref{Condition_estimatorRestrictions},
for the estimator $g_{T}\in\mathcal{L}\left(\bar{B}\right)$, for
any $\bar{B}\rightarrow\infty$ such that $\bar{B}=O\left(T^{1/2}\right)$,
$\left|g_{T}-g_{0}\right|_{\lambda,T}^{2}\lesssim\bar{B}\left(\frac{\ln K}{T}\right)^{1/2}$,
eventually, in probability. \end{corollary}

\subsubsection{Expansion in Terms of a Fixed Dictionary under $l_{1}$ Constraint\label{Section_Example_l1ExpansionFourierHolder}}

Consider the case of univariate functions with representation $f=\sum_{v=1}^{\infty}a_{v}e_{v}$
where $\left\{ e_{v}:v=1,2,...\right\} $ is a dictionary and $\sum_{v=1}^{\infty}\left|a_{v}\right|<\infty$.
Subspaces of such functions are considered in Barron et al. (2008).
A typical example is when $f$ is a polynomial. Then, let $\sum_{v=1}^{V}a_{v}e_{v}\left(x_{k}\right)$
be the (truncated) representation for the functions of the $k^{th}$
covariate for some finite $V$. Then, suppose that $g_{0}$ can be
written as 
\begin{equation}
g\left(x\right)=\sum_{k=1}^{K}b_{k}\sum_{v=1}^{V}a_{v_{k},k}e_{v_{k}}\left(x_{k}\right)\label{EQ_g_l1ExpansionExample}
\end{equation}
so that $\Theta_{k}=\left\{ e_{v}\circ\pi_{k}:v=1,2,...,V\right\} $
and $\sum_{k=1}^{K}\sum_{v=1}^{V}\left|b_{k}a_{v_{k},k}\right|\leq B_{0}$.
In this case, one can directly estimate the coefficients $b_{k}a_{v_{k},k}$
and reduce the optimization over $\Theta$ to the selection of an
element $e_{v}\circ\pi_{k}$ in $\Theta$. There are $V$ fixed elements
in each $\Theta_{k}$. Hence, the entropy integral for each $\Theta_{k}$
is a constant multiple of $\sqrt{\ln V}$. If no approximation error
is incurred (i.e. $g_{0}$ can be written as (\ref{EQ_g_l1ExpansionExample})),
then $\left|g_{T}-g_{0}\right|_{\lambda,T}^{2}\lesssim\left(\frac{\ln KV}{T}\right)^{1/2}$,
as in the linear case (Section \ref{Section_ExampleLinear}), but
with $KV$ variables instead of $K$. 

This framework adapts to sieve estimation of smooth functions, in
which case an approximation error is incurred. For definiteness suppose
that $\left\{ e_{v}:v=1,2,...\right\} $ are trigonometric polynomials
with period one, rather than a general dictionary. Let $\mathcal{H}$
be the class of Holder continuous functions on $\left[0,1\right]$
with exponent $\alpha>1/2$, constant one and uniformly bounded by
one, i.e., $\left|f\left(x\right)-f\left(y\right)\right|\leq\left|x-y\right|^{\alpha}$
and $\left|f\right|_{\infty}\leq1$, if $f\in\mathcal{H}$. By Bernstein
Theorem (e.g., Katznelson, 2002, p. 33), if $f\in\mathcal{H}$, there
is a finite absolute constant $c_{\alpha}$ depending only on $\alpha>1/2$
such that $f=\sum_{v=1}^{\infty}a_{v}e_{v}$ and $\sum_{v=1}^{\infty}\left|a_{v}\right|\leq c_{\alpha}$,
where the equality holds in the $\sup$ norm. Hence, in what follows,
we can take $\mathcal{H}$ to be equivalent to the class of functions
with such series expansion. Let $\mathcal{H}_{V}$ be the set of trigonometric
polynomials up to order $V$. By Jackson Theorem (e.g., Katznelson,
2002, p.49), for any $f\in\mathcal{H}$, there is a trigonometric
polynomial of order $V$, say $f_{V}\in\mathcal{H}_{V}$, such that
$\left|f_{V}-f\right|_{\infty}\lesssim V^{-\alpha}$. Suppose that
$g_{0}\in\mathcal{G}\left(1\right)$ (in (\ref{EQ_GTrueSet})), then,
using subscript $0$ to denote the coefficients of $g_{0}$, 
\[
g_{0}=\sum_{k=1}^{K}b_{k0}\left(\sum_{v=1}^{\infty}a_{vk0}e_{v}\right)=\sum_{k=1}^{K}\left(\bar{a}_{k0}b_{k0}\right)\left(\sum_{v=1}^{\infty}\left(\frac{a_{vk0}}{\bar{a}_{k0}}\right)e_{v}\right)
\]
setting $\bar{a}_{k0}:=\sum_{v=1}^{\infty}\left|a_{vk0}\right|$.
By the aforementioned remarks concerning Bernstein Theorem, there
is a finite constant $c_{\alpha}$ such that $\bar{a}_{k0}\leq c_{\alpha}$.
Hence, $\sum_{k=1}^{K}\left(\bar{a}_{k0}b_{k0}\right)\leq c_{\alpha}$,
using the constraint on the $b_{k0}$'s implied by restricting attention
to $g_{0}\in\mathcal{G}\left(1\right)$. Let $\Theta_{k}:=\left\{ \sum_{v=1}^{V}a_{v}e_{v}\circ\pi_{k}:\sum_{v=1}^{V}\left|a_{v}\right|\leq1\right\} $.
Using (\ref{EQ_approximationByUnivariate}) we can derive the approximation
error for this problem and deduce the following consistency rates. 

\begin{corollary}\label{Corollary_trigonometricPoly}Let $g_{0}\in\mathcal{G}\left(1\right)$
(as in \ref{EQ_GTrueSet}) with $\mathcal{H}$ Holder continuous with
exponent $\alpha>1/2$. Under Conditions \ref{Condition_stochasticRestrictions}
and \ref{Condition_estimatorRestrictions}, for $g_{T}\in\mathcal{L}\left(\bar{B}\right),$
there is a finite constant $c_{\alpha}$ such that $\left|g_{T}-g_{0}\right|_{\lambda,T}^{2}\lesssim\bar{B}\left(\frac{\ln KV}{T}\right)^{1/2}+V^{-2\alpha}+\max\left\{ c_{\alpha}-\bar{B},0\right\} ^{2}$
in probability. Hence, for any $\bar{B}\rightarrow\infty$, choosing
$V\asymp\left(T/\ln T\right)^{1/\left(4\alpha\right)}$, $\left|g_{T}-g_{0}\right|_{\lambda,T}^{2}\lesssim\bar{B}T^{-1/2}\left(\ln KT\right)^{1/2}$,
in probability.\end{corollary}

\subsubsection{Neural Networks }

Suppose $f\left(x\right)=\int_{\mathbb{R}}\int_{\mathbb{R}}\varphi\left(a_{1}x+a_{0}\right)d\nu\left(a_{0},a_{1}\right)$
for $x\in\left[0,1\right]$, where $\nu$ is a signed measure of finite
variation equal to $1/2$, and $\varphi$ is as in Section \ref{Section_thresholdModel}.
Up to a scaling constant, any continuous bounded function on $\left[0,1\right]$
admits this representation (e.g., Yukich et al., 1995, Section II).
Denote such class of univariate functions by $\mathcal{H}$. Usually,
$\varphi$ is a sigmoidal function, a monotone function such that
$\lim_{x\rightarrow\infty}\varphi\left(x\right)=1$ and $\lim_{x\rightarrow-\infty}\varphi\left(x\right)=0$,
e.g., the hyperbolic tangent $\tanh$. Consider the truncated series
expansion $\sum_{v=1}^{V}a_{1v}\varphi\left(a_{2v}x-a_{3v}\right)$
for some finite $V$. Denote the set of such series expansions with
$V$ terms by 
\[
\mathcal{H}_{V}:=\left\{ f\left(x\right)=\sum_{v=1}^{V}a_{1v}\varphi\left(a_{2v}x-a_{3v}\right):\sum_{v=1}^{V}\left|a_{1v}\right|\leq1,\,a_{2v},a_{3v}\in\mathbb{R}\right\} .
\]
Let $\Theta_{k}:=\left\{ f\circ\pi_{k}:f\in\mathcal{H}_{V}\right\} $.
Suppose that $g_{0}\in\mathcal{G}\left(B_{0}\right)$ (in (\ref{EQ_GTrueSet})).
The uniform error incurred by the best approximation in $\mathcal{H}_{V}$
for $\mathcal{H}$ is $V^{-1/2}$ $P$-almost surely (Theorem 2.1
in Yukich et al., 1995). Hence, using (\ref{EQ_approximationByUnivariate}),
the sieve with $V^{-1}=O\left(T^{-1/2}\right)$ leads to an approximation
error for $g_{0}$ that is $O\left(B_{0}T^{-1/4}\right)$. By the
arguments in Section \ref{Section_Example_l1ExpansionFourierHolder}
and the fact that $\varphi$ is Holder's continuous as in Section
\ref{Section_thresholdModel}, the following is deduced. 

\begin{corollary}\label{Corollary_neuralNet}Suppose that $g_{0}\in\mathcal{G}\left(B_{0}\right)$.
Under Conditions \ref{Condition_stochasticRestrictions} and \ref{Condition_estimatorRestrictions},
for the estimator $g_{T}\in\mathcal{L}\left(\bar{B}\right)$, for
any $V\geq1$, 
\[
\left|g_{T}-g_{0}\right|_{\lambda,T}^{2}\lesssim\bar{B}\left(\frac{\ln KV}{T}\right)^{1/2}+\max\left\{ B_{0}-\bar{B},0\right\} ^{2}+V^{-1}
\]
in probability. Hence, choosing $V\asymp T^{1/2}$, for any $\bar{B}\rightarrow\infty$,
$\left|g_{T}-g_{0}\right|_{\lambda,T}^{2}\lesssim\bar{B}\left(\frac{\ln KT}{T}\right)^{1/2}$
in probability.\end{corollary}

\subsubsection{Shape Constrained Estimator: Many Monotone Lipschitz Functions\label{Section_MonotoneBernstein}}

Consider estimation under monotone function constraints. Suppose $\mathcal{H}$
is the class of monotone increasing Lipschitz functions with domain
$\left[0,1\right]$ and bounded by one. Let the Lipschitz constant
be known and equal to $L$. Let $\mathcal{H}_{V}$ be the class of
univariate Bernstein polynomials of order $V$. Recall that $f_{V}$
is a Bernstein polynomial of order $V$ if $f_{V}\left(x\right)=\sum_{v=0}^{V}\binom{V}{v}a_{v}x^{v}\left(1-x\right)^{V-v}$,
$x\in\left[0,1\right]$, for any real $a_{v}$. If $a_{v}\geq a_{v-1}$
for all $v$'s, the polynomial is monotonically increasing. If also
$a_{v}-a_{v-1}\leq\alpha/V$ for all $v$'s, it is Lipschitz with
constant $\alpha$ (e.g., Lorentz, 1986, Ch.1.4). Hence, under these
constraints on the coefficients of the polynomial, $\mathcal{H}_{V}$
is a subset of functions with Lipschitz constant bounded by $\alpha$.
Moreover, for each $f\in\mathcal{H}$ there is an $f_{V}\in\mathcal{H}_{V}$
such that $\left|f_{V}-f\right|_{\infty}\lesssim\alpha V^{-1/2}$
(e.g., Lorentz, 1986, Theorem 1.6.1). Let $\Theta_{k}:=\left\{ f\circ\pi_{k}:f\in\mathcal{H}_{V}\right\} $.
Estimation of monotone functions with known Lipschitz constraint can
be conveniently performed by Bernstein polynomials, using the algorithm
in Section \ref{Section_EstimationDetails}. The estimation problem
becomes a linear programming problem at each step. To see this, define
$q_{v}\left(x\right):=\binom{V}{v}x^{v}\left(1-x\right)^{V-v}$. In
particular, $D_{T}\left(g,\theta\right)$ in Section \ref{Section_EstimationDetails}
is linear in $\theta$. Hence, maximization of $D_{T}\left(F_{j-1},\theta\right)$
w.r.t. $\theta\in\Theta_{k}$ is equivalent to 
\[
\max_{\left\{ a_{v}:v\leq V\right\} }\sum_{v=0}^{V}a_{v}\left[\int_{0}^{T}q_{v}\left(X_{k}\left(t\right)\right)dN\left(t\right)-\int_{0}^{T}q_{v}\left(X_{k}\left(t\right)\right)\exp\left\{ g\left(X\left(t\right)\right)\right\} dt\right]
\]
such that $0\leq a_{v-1}\leq a_{v}\leq1$, and $a_{v}-a_{v-1}\leq\alpha/V$,
$v=1,2,...,V$. This is routinely solved by the simplex method for
each $k$. Further, details concerning the estimation procedure can
be deduced along these lines. From Corollary 2.7.2 in van der Vaart
and Wellner (2000) deduce that the entropy integral for functions
in $\mathcal{H}_{V}$ is a constant multiple of $\alpha^{1/2}$. The
following uses this observation when applying Theorem \ref{Theorem_consistency}.

\begin{corollary}\label{Corollary_monotoneBernstein}Let $g_{0}\in\mathcal{G}\left(B_{0}\right)$.
Under Conditions \ref{Condition_stochasticRestrictions} and \ref{Condition_estimatorRestrictions},
for $g_{T}\in\mathcal{L}\left(\bar{B}\right),$ $V\gtrsim\alpha^{3/2}\times T^{1/2}$,
and $\bar{B}\rightarrow\infty$, $\left|g_{T}-g_{0}\right|_{\lambda,T}^{2}\lesssim\bar{B}\left(\frac{\alpha+\ln K}{T}\right)^{1/2}$,
in probability. \end{corollary}

If the Lipschitz constant is not known, we can let $\alpha\rightarrow\infty$
in the estimation. In this case, the entropy integral is finite, but
not bounded. 

\subsection{Choice of $\bar{B}$\label{Section_chooseB}}

Given the relation with $l_{1}$ penalization (see (\ref{EQ_lassoPenalised})),
the model degrees of freedom can be approximated by the resulting
number of active variables (e.g., Bradic et al., 2011). Hence, the
value $\bar{B}$ can chosen by maximizing the Akaike's penalized likelihood
(AIC): $AIC_{T}\left(B\right):=\sup_{g\in\mathcal{L}\left(B\right)}L_{T}\left(g\right)-K_{B}$
where $K_{B}$ is the number of nonzero parameters in $g_{T}=\arg\sup_{g\in\mathcal{L}\left(B\right)}$.
This is less computationally intensive than cross-validation. (In
a time series context, cross-validation requires some care except
for a special few cases; e.g., Burman et al., 1994). 

For very large sample size, AIC will select models that are very large.
In this case cross-validation with a large validation sample (i.e.,
leaving out a large proportion of the data) tends to select smaller
models. Hence, the method to be used depends on the context. See Sections
\ref{Section_ComputationDetailsEmpirical} and \ref{Section_simulations}
for further discussion and applications. Finally, note that to speed
up the calculations for the choice of $\bar{B}$, the algorithm in
Section \ref{Section_EstimationDetails} can be used without line
search.

\subsection{Model Fit and Out of Sample Evaluation\label{Section_outOfSampleEvaluation}}

Model adequacy can be carried out in large samples using the log-likelihood
evaluated out of sample. The out of sample log-likelihood ratio for
two competing models $g_{t},g_{t}'\in\mathcal{L}\left(\bar{B}\right)$
which are predictable at time $t$ is 
\[
L_{S}\left(g,g'\right)=\int_{0}^{S}\left[g_{t}\left(X\left(t\right)\right)-g_{t}'\left(X\left(t\right)\right)\right]dN\left(t\right)-\int_{0}^{S}\left[\exp\left\{ g_{t}\left(X\left(t\right)\right)\right\} -\exp\left\{ g_{t}'\left(X\left(t\right)\right)\right\} \right]dt.
\]
In practice, one may split the sample and estimate $g_{t}$ and $g_{t}'$
on the first half, or every so often using past observations. The
predictable part of the log-likelihood ratio is 
\[
H_{S}\left(g,g'\right)=\int_{0}^{S}\left[g_{t}\left(X\left(t\right)\right)-g_{t}'\left(X\left(t\right)\right)\right]d\Lambda\left(t\right)-\int_{0}^{S}\left[\exp\left\{ g_{t}\left(X\left(t\right)\right)\right\} -\exp\left\{ g_{t}'\left(X\left(t\right)\right)\right\} \right]dt,
\]
where $\Lambda\left(t\right)$ is a short for $\Lambda\left(\left[0,t\right]\right)$.
Model $g$ outperforms $g'$ if $H_{S}\left(g,g'\right)>0$. (If $g=g_{0}$,
$H_{S}\left(g,g'\right)\geq0$, with equality only if $g'=g_{0}$,
see Lemma \ref{Lemma_predictableLR} in the supplementary material.)
The following null hypothesis can be tested: $H_{S}\left(g,g'\right)=0$
against a one or two sided alternative. Under the null, 
\[
L_{S}\left(g,g'\right)=\int_{0}^{S}\left[g_{t}\left(X\left(t\right)\right)-g_{t}'\left(X\left(t\right)\right)\right]d\left(N\left(t\right)-\Lambda\left(t\right)\right).
\]
The following martingale result is the justification for the testing
procedure. 

\begin{proposition}\label{Proposition_outOfSampleTest}Suppose that
$g_{t}$ and $g_{t}'$ are predictable bounded processes and $H_{S}\left(g,g'\right)=0$.
Suppose that as $S\rightarrow\infty$ 
\[
\frac{1}{S}\int_{0}^{S}\left[g_{t}\left(X\left(t\right)\right)-g_{t}'\left(X\left(t\right)\right)\right]^{2}d\Lambda\left(t\right)\rightarrow\sigma^{2}>0
\]
in probability. Let $\hat{\sigma}_{S}^{2}:=\frac{1}{S}\int_{0}^{S}\left[g_{t}\left(X\left(t\right)\right)-g_{t}'\left(X\left(t\right)\right)\right]^{2}dN\left(t\right)$.
Then, $L_{S}\left(g,g'\right)/\sqrt{S\hat{\sigma}_{S}^{2}}$ converges
in distribution to a standard normal random variable.\end{proposition}

The testing framework falls within the prequential framework of Dawid
(e.g., Seillier-Moiseiwitsch and Dawid, 1993, for applications). 

This methodology can be applied in various ways. As an example, consider
a sample of size $2T$. Use $\left[0,T\right]$ to find the estimators
$g_{T}$ and $g_{T}'$. Conduct the test on $(T,2T]$ so that, mutatis
mutandis, $S=T$ in the proposition. In this case, $g_{T}$ and $g_{T}'$
are predictable. We need to suppose that the testing sample size $S$
increases to infinity in order to apply the result. If the size $T$
of the testing sample is large, the asymptotic result is applicable.

\section{Application to Estimation and Forecasting of Trade Arrivals of New
Zealand Dollar Futures\label{Section_EmpiricalApplication}}

One motivation for the estimation method discussed here was to understand
the variables that affect the trade arrivals of the New Zealand dollar
futures, i.e. the futures on NZDUSD traded on the Chicago Mercantile
Exchange (CME). The New Zealand dollar is a liquid currency futures,
but not as much as other currency futures (Fx futures) such as the
Euro, Australian dollar and the Swiss Franc (against the dollar).
What are the variables that affect a trade arrival, say a buy trade?
Are these variables, and relations if any, stable in the sense that
one can forecast a buy trade arrival tomorrow having estimated a model
with today's data? These questions are important to the understanding
of market microstructures, and the general etiology of the Fx futures
markets and its relation to other instruments, e.g., equity markets,
commodities etc. In fact, the New Zealand dollar belongs to the commodity
Fx group that includes, for example, Australian dollar, Canadian dollar.
These are the currencies of countries whose economy relies on commodity
exports. Anecdotal evidence seems to suggest that the New Zealand
dollar tends to increase in value when risk appetite increases.

Below, the data are described and subsequently the model is estimated.

\subsection{Data and Variables Description }

The estimation of the intensity of trade arrivals is an important
problem (e.g., Hall and Hautsch, 2007). New Zealand Dollar futures
(the NZD/USD futures front month contract, whose ticker is 6N) are
traded on the Chicago Mercantile Exchange. Two days of trading between
8am to 5pm GMT are considered, in particular, 10/09/2013-11/09/2013.
The time slot is based on liquidity considerations. Data are proprietary
and were collected with high precision time stamps by a Chicago proprietary
trading firm with co-located servers in the Aurora data center in
Chicago. In consequence, trades were classified as buy or sell with
minimal probability of error. The data have nanosecond time stamps,
and trades time stamps have been adjusted to account for delays in
the CME network and reporting (these adjustments are in the order
of half a millisecond). This ensures that only information prior to
the trade is used to define covariates. Buy and sell intensities are
estimated separately. The covariates are derived using information
from 6N as well as from other contracts that are perceived as likely
to have an impact. 

Covariates are constructed from the following CME futures: NZDUSD
(6N), AUDUSD (6A), EURUSD (6E), GBPUSD (6B), CADUSD (6C), JPYUSD (6J),
CHFUSD (6S), MXNUSD (6M), Crude Oil (CL), Gold (GC) and mini S\&P500
(ES). For each instrument, covariates were derived from order book
and trade updates. In particular, the variables are mid-price returns,	bid-ask
spread, 	volume imbalances for the first two levels, trade imbalances
and trade duration. Variables are updated every time there is a change
in their value. For example, the return is computed when there is
a mid price change from the previous mid. Volume imbalances are computed
as the difference of the bid and ask quantities on each level (the
contracts usually quote prices for 5 levels). These differences are
then standardized by the sum on the bid and ask quantity on that level.
Trade imbalances are the signed traded size, positive if a buy and
negative if a sell. Excluding the spread, moving averages of all variables
are also computed. In particular, moving averages of order 1, 2, 4,
8, 16, 32, 64, 128 are used. This is to allow information at slightly
different frequencies to affect the intensity in a way similar to
MIDAS. Overall, the total number of variables is 508 including a constant.
A model that also allows for squares and third powers of all the standardized
variables is also estimated. In this case, the total number of variables
is 1522 including a constant. Once the feature variables are computed,
in order to reduce the computational burden, these are sampled only
when there is an update in the NZDUSD futures. The argument is that
if an instrument leads 6N, then the book for 6N would update before
a trade. 

\subsection{Computational Details\label{Section_ComputationDetailsEmpirical}}

The two-day sample is split into three parts. The first half of day
one is the estimation sample. The second half of day one is the validation
sample. The second day is the testing sample. The variables are windsorized
at the 95\% quantile and then standardized by it so as to take values
in $\left[-1,1\right]$. For estimation of the cubic polynomial, powers
of the variables are computed after having mapped the variables into
$\left[-1,1\right]$. The quantile is computed using the data from
day one only. Hence windsorization on the testing sample is based
on the previous day 95\% quantile. After windsorization, the set of
weights $\mathcal{W}$ is chosen equal to the sample estimator of
the $L_{2}$ norm, i.e. $w_{\theta}=\left(\frac{1}{T}\int_{0}^{T}\theta^{2}\left(X\left(t\right)\right)dt\right)^{1/2}$
over day one. This ensures that all variables are given the same importance.
The model is estimated for $B\in\left\{ 2,4,8,16\right\} $ on the
estimation sample. We set $\bar{B}$ equal to the $B$ that maximizes
the likelihood on the validation sample. This method is preferred
to AIC as here the sample size is very large. With this choice of
$\bar{B}$, the model is then re-estimated using the data in the first
day, i.e. both estimation and validation sample. This approach is
feasible in a large sample and avoids some of the drawbacks of cross-validation
for dependent observations.

\subsection{Estimation Results\label{Section_empiricalEstimation}}

It is difficult to clearly and concisely report the variables that
appear to be most important for the intensity. In fact, though with
some small coefficients, a large number of variables are included
by the method described here. For the linear model, the chosen $\bar{B}$
results in a model for buy and sell trades with 77 and 68 covariates,
respectively. For the cubic case, the number was slightly larger.
Including many variables with relatively small coefficients produces
an averaging effect across many variables and can provide a hedge
against instability and noise, in a way similar to forecast combination.

The intersection of the first ten variables in the linear model for
buy and sell trades is reported in Table \ref{Table_empiricallyImportantVariables}.
These variables can be seen as some form of more stable subset of
variables (Meinshausen and Bühlmann, 2010, for formal methods on stability
selection). 

\begin{longtable}{clc}
\caption{ Most important variables affecting buy and sell trade arrival in
linear model.}
\label{Table_empiricallyImportantVariables}\tabularnewline
\cline{1-2} 
\endfirsthead
Instrument & Variable & \tabularnewline
\cline{1-2} 
6N & Volume Imbalance on Level 1 & \tabularnewline
6N & Volume Imbalance on Level 2 & \tabularnewline
6N & Spread & \tabularnewline
6A & Duration from Last Trade & \tabularnewline
 &  & \tabularnewline
\end{longtable}

Interestingly, past durations of 6N (the New Zealand Dollar futures)
do not seem to be as important sot they are not in Table \ref{Table_empiricallyImportantVariables}.
However, the durations of 6A (the Australian Dollar) appear to be
important. The Australian Dollar tends to correlate with the New Zealand
Dollar, but it is more liquid. Hence, it might provide useful information
on trade arrival. Past durations have been found to be important predictors
in some high frequency financial applications (e.g., Engle and Russell,
1998). However, book information seems to have greater impact. In
the next section, a linear model using only the variables in Table
\ref{Table_empiricallyImportantVariables} will also be used for comparison
and will be referred to as the restricted linear model.

\subsubsection{Out of Sample Performance}

Having estimated the model on the first day, it is of interest to
see if the model can be used to explain a trade arrival out of sample.
This is done computing the average log-likelihood ratio $L_{S}\left(g,g'\right)/S$,
and $\hat{\sigma}_{S}/\sqrt{S}$ on the second day (see Proposition
\ref{Proposition_outOfSampleTest}). Confidence intervals can then
be constructed using Proposition \ref{Proposition_outOfSampleTest}.
The goal is to assess the out of sample performance of the linear
and cubic model as well as the restricted model (the one with variables
in Table \ref{Table_empiricallyImportantVariables}). It is of interest
to verify if restricting attention to a linear model might produce
similar out of sample results. When comparing to the constant intensity
(Conts.), the constant is computed as the out of sample maximum likelihood
estimator, i.e., the best constant intensity with hindsight.

Table \ref{Table_empiricalResults} shows that all the models do improve
on the constant intensity with overwhelming evidence. When looking
at the relative merits of the unrestricted models, it becomes unclear
whether a cubic model adds value out of sample. Looking at the restricted
linear model relative to the unrestricted linear one, there is overwhelming
evidence that the unrestricted model should be preferable. It is interesting
that when comparing the restricted models, there is overwhelming evidence
that a cubic model does improve on the linear one. From these results
one could infer that modelling nonlinearities does pay off when looking
at small dimensional models. However, when models are linear, but
with many covariates, nonlinear impact of book and trade variables
is less obvious. The simulation results of Section \ref{Section_simulations}
support this claim. 

\begin{longtable}{lccccc}
\caption{Out of sample performance of models: $g$ vs. $g'$ with $g$ and
$g'$ as defined in the headings below.}
\label{Table_empiricalResults}\tabularnewline
\cline{1-5} 
\endfirsthead
 & \multicolumn{2}{l}{Lin. vs. Const.} & \multicolumn{2}{l}{Cubic vs. Const. } & \tabularnewline
\cline{1-5} 
 & Buy & Sell & Buy  & Sell & \tabularnewline
Avg.Log-LR.$\times10^{2}$ & 3.77 & 4.48 & 4.02 & 4.56 & \tabularnewline
S.E.$\times10^{2}$ & 0.33 & 0.25 & 0.31 & 0.26 & \tabularnewline
P-Val. & <0.01 & <0.01 & <0.01 & <0.01 & \tabularnewline
\cline{1-5} 
 & \multicolumn{2}{l}{Cubic vs. Linear} & \multicolumn{2}{c}{} & \tabularnewline
\cline{1-5} 
 & Buy  & Sell &  &  & \tabularnewline
Avg.Log-LR.$\times10^{2}$ & 0.25 & 0.08 &  &  & \tabularnewline
S.E.$\times10^{2}$ & 0.08 & 0.07 &  &  & \tabularnewline
P-Val. & <0.01 & 0.22 &  &  & \tabularnewline
\cline{1-5} 
 & \multicolumn{2}{l}{Lin. Restr. vs. Const. } & \multicolumn{2}{l}{Lin. Restr. vs. Lin.} & \tabularnewline
\cline{1-5} 
 & Buy  & Sell & Buy  & Sell & \tabularnewline
Avg.Log-LR.$\times10^{2}$ & 1.14 & 1.24 & -2.63 & -3.25 & \tabularnewline
S.E.$\times10^{2}$ & 0.15 & 0.14 & 0.20 & 0.19 & \tabularnewline
P-Val. & <0.01 & <0.01 & <0.01 & <0.01 & \tabularnewline
\cline{1-5} 
 & \multicolumn{2}{l}{Cubic Restr. vs. Lin. Restr.} &  &  & \tabularnewline
\cline{1-5} 
 & Buy  & Sell &  &  & \tabularnewline
Avg.Log-LR.$\times10^{2}$ & 0.26 & 0.25 &  &  & \tabularnewline
S.E.$\times10^{2}$ & 0.10 & 0.06 &  &  & \tabularnewline
P-Val. & <0.01 & <0.01 &  &  & \tabularnewline
\end{longtable}

\section{Numerical Examples\label{Section_simulations} }

As remarked in Section \ref{Section_likelihood}, $\left\{ \Lambda\left((T_{i-1},T_{i}]\right):i\in\mathbb{N}\right\} $
($\Lambda$ as in (\ref{EQ_intensityRepresentation})) is i.i.d. exponentially
distributed with mean 1. For simplicity, in the simulations, it is
assumed that the covariates only update at the jump times $T_{i}$'s.
Hence, the intervals $(T_{i-1},T_{i}]$ are simulated from an exponential
distribution with parameter $\exp\left\{ g\left(X\left(T_{i-1}\right)\right)\right\} $,
i.e., with mean $\exp\left\{ -g\left(X\left(T_{i-1}\right)\right)\right\} $.
The covariates are standard Gaussian random variables with Toeplitz
covariance $Cov\left(X_{k}\left(t\right),X_{l}\left(t\right)\right)=\rho^{\left|k-l\right|}$
and uncorrelated over time. The variables have been capped to $2$
in absolute value, i.e., they take values in $\left[-2,2\right]$. 

The parameters in the simulation are $K\in\left\{ 10,50\right\} $
number of covariates, $T=T_{100}$ (recall $N\left(T_{n}\right)=n$)
sample size, and $\rho\in\left\{ 0,0.75\right\} $. Different choices
of $g_{0}$, and $\Theta$ are considered. These are summarized as
follows. For estimation simplicity, $\Theta$ is a finite set of functions. 

\subsection{True Unknown Model $g_{0}$\label{Section_trueModelSimulation}}

Here we describe various options for the true function $g_{0}$. The
true function $g_{0}$ takes the form $g_{0}\left(x\right)=\sum_{k=1}^{K}g_{0}^{\left(k\right)}\left(x\right)$,
where the functions $g_{0}^{\left(k\right)}$ are defined as follows. 

\paragraph{True additive functions. }

Linear: $g_{0}^{\left(k\right)}\left(x\right)=b_{0k}x_{k}$; NonLinear:
$g_{0}^{\left(k\right)}\left(x\right)=b_{0k}\left(\left|x_{k}\right|+0.5x_{k}\right)$.

\paragraph{Active variables. }

FewLarge $b_{0k}=1$ for $k=1,2,3$, $b_{0k}=0$ for $k>3$; ManySmall
$b_{0k}=1/\sqrt{10}$ for $k\leq10$, $b_{0k}=0$ for $k>10$. Even
when there is no model misspecification, these values are unknown
to the researcher. 

\subsection{Estimator in $\mathcal{L}\left(B,\Theta,\mathcal{W}\right)$}

Here we define the parameter space $\mathcal{L}\left(B,\Theta,\mathcal{W}\right)$
used by the researcher. Estimation is carried out allowing for model
misspecification. Hence, depending on the design, the choice of functions
does not need to correspond to the true functions $g_{0}^{\left(k\right)}$
(Section \ref{Section_trueModelSimulation}). The estimated models
are of the form $g\left(x\right)=\sum_{k=1}^{K}\sum_{\theta\in\Theta_{k}}b_{\theta}\theta\left(x\right)$.
Details regarding $\Theta_{k}$ and the estimation of the $b_{\theta}$'s
are as follows. 

\paragraph{Functions in $\Theta$.}

Linear (Lin): $\theta\left(x\right)=x_{k}$ for $\theta\in\Theta_{k}$;
Monomials (Poly): $\theta\left(x\right)=\left(x_{k}/2\right)^{a}$
for $\theta\in\Theta_{k}$ with $a=1,2,3$. A constant is added by
default in the estimations. When the true function is linear (i.e.,
$g_{0}^{\left(k\right)}\left(x\right)=b_{0k}x_{k}$) there is no misspecification
error. However, the coefficients still need to be estimated, many
of which can be zero. When the true function is nonlinear, misspecification
error will be incurred even when estimation is carried out using a
polynomial (Poly). However, in this case, the degree of misspecification
will be small. 

\paragraph{Choice of $\bar{B}$ and $\mathcal{W}$ }

The parameter $\bar{B}$ is chosen as the $B\in\left\{ 1,4,8,16\right\} $
that maximizes $AIC_{T}$ as defined in Section \ref{Section_chooseB}.
In this case, the sample size is relatively small and the performance
of $AIC_{T}$ and cross-validation (leaving out many variables) was
similar. Hence, $AIC_{T}$ is preferred for computational convenience.
We applied the algorithm in Section \ref{Section_EstimationDetails}
with $F_{0}=\ln\left(N\left(T\right)/T\right)$ rather than $F_{0}=0$.
In this case, $e^{F_{0}}$ is an estimator of $P\lambda$, the expected
intensity. The main reason was to reduce fine tuning of the set of
possible values of $B$ to the different functions and simulation
designs. The simulation design is such that as the number of active
variables increases, $P\lambda$ increases and in consequence $B$. 

The weights in $\mathcal{W}$ are chosen to be the sample $L_{2}$
norm as in Section \ref{Section_ComputationDetailsEmpirical}. Note
that no winsorization is applied to the variables, as they are already
bounded.

\subsection{Simulation Results \label{Section_simulationResultsIID}}

The following loss function is considered to assess the model fit,

\begin{equation}
Loss\left(g\right):=\frac{\int_{T}^{T+S}\left[g_{0}\left(X\left(t\right)\right)-g\left(X\left(t\right)\right)\right]^{2}dN\left(t\right)}{\int_{T}^{T+S}\left[g_{0}\left(X\left(t\right)\right)-\gamma_{0}\right]^{2}dN\left(t\right)}\label{EQ_LossSimulation}
\end{equation}
where $\gamma_{0}:=\frac{\int_{T}^{T+S}g_{0}\left(X\left(t\right)\right)dN\left(t\right)}{N\left(T+S\right)-N\left(T\right)}$.
This loss function is justified noting that when $S$ is large, $Loss\left(g\right)\simeq\left|g_{0}-g\right|_{\lambda}^{2}/\left[\inf_{\gamma>0}\left|g_{0}-\gamma\right|_{\lambda}^{2}\right]$.
Hence, the numerator in $Loss$ is an approximation to the convergence
criterion of Theorem \ref{Theorem_consistency}, while the denominator
is the error incurred by $\gamma_{0}$, the best constant approximation
with hindsight. The standardization ensures that $Loss\left(g\right)\in[0,1)$
if $g$ improves over $\gamma_{0}$, if not $Loss\left(g\right)\geq1$.
The denominator in $Loss$ is the benchmark for the finite sample
experiment carried out here. In the simulations, data are generated
for a sample period $\left[0,T_{1100}\right]$, and the model is estimated
on $\left[0,T_{100}\right]$ and out of sample performance is evaluated
on $\left[T_{100},T_{1100}\right]$. Hence, in $Loss$, $T=T_{100}$
and $S=T_{1100}-T_{100}$. Table 3 reports the median of $Loss\left(g_{T_{100}}\right)$
(LOSS) together with the 75\% and 25\% quantile. 

Overall, different choices of true model (linear or convex) and basis
functions allow us to gauge the main features of the estimator. The
results in Table \ref{Table_simulationResults} can be summarized
as follows. There is a clear advantage in using a nonlinear model
when the true model is nonlinear, but also a considerable loss (mostly
due to estimation error) when the true model is linear. For nonlinear
estimators such as polynomials, a judicious choice of $\mathcal{W}$
to dump the effect of higher order coefficients can make the estimator
more robust. The present choice of $\mathcal{W}$ is equivalent to
standardizing the variables by their $L_{2}$ norm. This is simple,
but might lead to big oscillations if the order of polynomial is not
as small as it is here. Choice of $\mathcal{W}$ is an important part
of the modelling and estimation procedure when dealing with polynomials.
An increase in variables correlation produces better forecasts. This
is in contrast with the problem of variable screening. Numerical experiments
of the author - not reported here - as well as related results in
the literature (e.g., Bradic et al., 2011) show that, in this context,
false discovery of active variables increases substantially with correlation.
This is natural, as correlation confounds the merits of each single
variable. The forecasting and variable screening are related, but
complementary problems, which require a separate treatment. 

\begin{longtable}{cccccccccc}
\caption{Simulation results relative to the best constant intensity with hindsight.
Estimation is based on samples of size $T_{100}$ corresponding to
$N\left(T_{100}\right)=100$ number of jumps. The table reports the
median (Med.) and the 25 (Q25\%) and 75 (Q75\%) per cent quantiles
of $Loss\times100$ ($Loss$ as in (\ref{EQ_LossSimulation})). A
number below 100 means a relative improvement on the best constant
intensity with hindsight.}
\label{Table_simulationResults}\tabularnewline
\endfirsthead
 & \multicolumn{3}{c}{Loss$\times$100} &  &  & \multicolumn{3}{c}{Loss$\times$100} & \tabularnewline
 & Med. & Q25\% & Q75\% &  &  & Med. & Q25\% & Q75\% & \tabularnewline
 & \multicolumn{3}{c}{$\rho=0$} &  &  & \multicolumn{3}{c}{$\rho=0.75$} & \tabularnewline
\cline{1-9} 
 & \multicolumn{8}{c}{$g_{0}$ is Linear FewLarge $K=10$} & \tabularnewline
\cline{1-9} 
Lin & 3.70 & 2.37 & 5.72 &  &  & 1.69 & 1.11 & 2.61 & \tabularnewline
Poly & 5.54 & 3.60 & 8.10 &  &  & 2.76 & 1.85 & 4.19 & \tabularnewline
\cline{1-9} 
 & \multicolumn{8}{c}{$g_{0}$ is Linear FewLarge $K=50$} & \tabularnewline
\cline{1-9} 
Lin & 6.83 & 4.92 & 9.41 &  &  & 3.64 & 2.34 & 5.26 & \tabularnewline
Poly & 10.61 & 8.01 & 13.66 &  &  & 4.30 & 2.82 & 6.65 & \tabularnewline
\cline{1-9} 
 & \multicolumn{8}{c}{$g_{0}$ is Linear ManySmall $K=10$} & \tabularnewline
\cline{1-9} 
Lin & 13.42 & 9.90 & 19.16 &  &  & 2.72 & 1.93 & 3.87 & \tabularnewline
Poly & 32.88 & 24.93 & 41.65 &  &  & 4.05 & 2.63 & 5.90 & \tabularnewline
\cline{1-9} 
 & \multicolumn{8}{c}{$g_{0}$ is Linear ManySmall $K=50$} & \tabularnewline
\cline{1-9} 
Lin & 47.02 & 35.29 & 57.26 &  &  & 4.81 & 3.42 & 6.22 & \tabularnewline
Poly & 60.83 & 51.45 & 74.57 &  &  & 6.23 & 4.64 & 8.30 & \tabularnewline
\cline{1-9} 
 & \multicolumn{8}{c}{$g_{0}$ is Convex FewLarge $K=10$} & \tabularnewline
\cline{1-9} 
Lin & 81.08 & 75.13 & 92.10 &  &  & 70.08 & 63.90 & 77.56 & \tabularnewline
Poly & 19.92 & 15.03 & 26.82 &  &  & 9.35 & 7.19 & 12.77 & \tabularnewline
\cline{1-9} 
 & \multicolumn{8}{c}{$g_{0}$ is Convex FewLarge $K=50$} & \tabularnewline
\cline{1-9} 
Lin & 110.23 & 90.67 & 123.28 &  &  & 83.53 & 72.46 & 95.66 & \tabularnewline
Poly & 35.47 & 28.08 & 45.36 &  &  & 14.70 & 11.86 & 19.36 & \tabularnewline
\cline{1-9} 
 & \multicolumn{8}{c}{$g_{0}$ is Convex ManySmall $K=10$} & \tabularnewline
\cline{1-9} 
Lin & 97.95 & 87.20 & 112.47 &  &  & 73.59 & 66.67 & 82.64 & \tabularnewline
Poly & 17.16 & 14.42 & 20.58 &  &  & 5.49 & 4.60 & 6.90 & \tabularnewline
\cline{1-9} 
 & \multicolumn{8}{c}{$g_{0}$ is Convex ManySmall $K=50$} & \tabularnewline
\cline{1-9} 
Lin & 104.12 & 94.80 & 114.59 &  &  & 67.38 & 62.55 & 74.41 & \tabularnewline
Poly & 48.24 & 40.72 & 57.95 &  &  & 10.67 & 8.86 & 13.13 & \tabularnewline
 &  &  &  &  &  &  &  &  & \tabularnewline
\end{longtable}

\subsubsection{Simulations with Dynamics: Hawkes Process with Covariates\label{Section_SimulationHawkesManyCovariates}}

The previous simulations considered time independent covariates. Here,
we make the covariates time dependent, following an autoregressive
process and also allow the intensity to follow a Hawkes process. Consider
the intensity 
\begin{equation}
\lambda\left(t\right)=\exp\left\{ \ln\left(c_{0}+\int_{\left(0,t\right)}e^{-a_{0}\left(t-s\right)}dN\left(s\right)\right)+g_{0}\left(X\left(t\right)\right)\right\} \label{EQ_hawkesExpCov}
\end{equation}
This is in the form of Section \ref{Section_hawkesProcess}, though
the function $f\left(\cdot\right)=\ln\left(c_{0}+\cdot\right)$ is
bounded below (because its domain is positive), it is not bounded
above. Here, $c_{0}>0$ is required to avoid degeneracy. To directly
apply the results in Section \ref{Section_hawkesProcess} we could
use $f\left(\cdot\right)=\max\left\{ \ln\left(c_{0}+\cdot\right),\bar{c}\right\} $
instead, for some finite $\bar{c}$, in which case the process is
assured to be stationary (see Corollary \ref{Lemma_hawkesApproximation}).
The process simplifies to 
\begin{equation}
\lambda\left(t\right)=\left(c_{0}+\int_{\left(0,t\right)}e^{-a_{0}\left(t-s\right)}dN\left(s\right)\right)\exp\left\{ g_{0}\left(X\left(t\right)\right)\right\} .\label{EQ_hawkesExpCov2}
\end{equation}
Using results for marked Hawkes processes (e.g., Bremaud et al., 2002)
one could conjecture that (\ref{EQ_hawkesExpCov2}) would be stationary
if $a_{0}>\mathbb{E}\exp\left\{ g_{0}\left(X\left(t\right)\right)\right\} $.
To the author's knowledge, formal existing results do not fit exactly
into the framework of (\ref{EQ_hawkesExpCov2}). In the simulations
we add a constant to the true model, i.e., $g_{0}\left(x\right)=\gamma+\sum_{k=1}^{K}g_{0}^{\left(k\right)}\left(x\right)$
where $\gamma=-\mathbb{E}\exp\left\{ \sum_{k=1}^{K}g_{0}^{\left(k\right)}\left(X\left(t\right)\right)\right\} $,
so that $\mathbb{E}\exp\left\{ g_{0}\left(X\left(t\right)\right)\right\} =1$.
This should ensure the aforementioned stationarity of (\ref{EQ_hawkesExpCov2})
when $a_{0}>1$. Other than that, the true models for $g_{0}$ are
as in Section \ref{Section_trueModelSimulation}. In the simulations
we verified that the term in parenthesis on the r.h.s. of (\ref{EQ_hawkesExpCov2})
remains bounded, hence ensuring stationarity with no need of a capping
constant $\bar{c}$. This model can be simulated and estimated and
details concerning this and some of the calculations to be discussed
below can be found in Section \ref{Section_additionalHawkesSimulationDetails}
of the supplementary material. 

As in the previous simulation, we let $X\left(t\right)=X\left(T_{i-1}\right)$
for $t\in(T_{i-1},T_{i}]$. However, the $X\left(T_{i}\right)$'s
now follow the vector autoregression $X\left(T_{i}\right)=0.95X\left(T_{i-1}\right)+\varepsilon_{i}$,
$X\left(T_{0}\right)=\varepsilon_{0}$, where the $K$ dimensional
innovations $\varepsilon_{i}$'s are generated as the i.i.d. truncated
Gaussian with Toeplitz covariance exactly as the i.i.d. $X\left(T_{i}\right)$'s
used in Section \ref{Section_simulationResultsIID}. If the $X\left(T_{i}\right)$'s
were independent as in the previous simulation, the dependence in
the Hawkes component would be confounded by the independent variability
in $\exp\left\{ g_{0}\left(X\left(T_{i}\right)\right)\right\} $.
Given the dependence structure, we use a larger sample size $T_{n}$
with $n=200$. In the simulations, we set $c_{0}=2$ and $a_{0}=1.3$. 

Except for these differences, the set up is the same as in the previous
simulation. However, we have $c_{0}$ and $a_{0}$ as extra parameters
to be estimated. The goal of the simulations is to see how the remarks
made in the case of time independent variables may hold in this case.
Results are reported in Table \ref{Table_simulationResultsHawkes}.
Results in Table \ref{Table_simulationResults} and Table \ref{Table_simulationResultsHawkes}
are not directly comparable, because of the scaling required for stationarity.
However, we can establish conclusions in relative terms. 

Table \ref{Table_simulationResultsHawkes} confirms the overall situation
of Table \ref{Table_simulationResults}. However, dependence makes
the problem harder, as expected. The relative benefit of estimating
a nonlinear model when the true $g_{0}$ is nonlinear decreases substantially
in the present scenario. For example, in the case of Convex ManySmall
$K=50$, the ratio of the loss for Lin and Poly in Table \ref{Table_simulationResults}
is $104.12/48.24=2.16$, while in Table \ref{Table_simulationResultsHawkes}
is $52.55/42.55=1.23$. 

\begin{longtable}{cccccccccc}
\caption{Simulation results relative to the best constant intensity with hindsight.
The model is as in (\ref{EQ_hawkesExpCov}). Estimation is based on
samples of size $T_{200}$ corresponding to $N\left(T_{200}\right)=200$
number of jumps. The table reports the median (Med.) and the 25 (Q25\%)
and 75 (Q75\%) per cent quantiles of $Loss\times100$ ($Loss$ as
in (\ref{EQ_LossSimulation})). A number below 100 means a relative
improvement on the best constant intensity with hindsight.}
\label{Table_simulationResultsHawkes}\tabularnewline
\endfirsthead
 & \multicolumn{3}{c}{Loss$\times$100} &  &  & \multicolumn{3}{c}{Loss$\times$100} & \tabularnewline
 & Med. & Q25\% & Q75\% &  &  & Med. & Q25\% & Q75\% & \tabularnewline
 & \multicolumn{3}{c}{$\rho=0$} &  &  & \multicolumn{3}{c}{$\rho=0.75$} & \tabularnewline
\cline{1-9} 
 & \multicolumn{8}{c}{$g_{0}$ is Linear FewLarge $K=10$} & \tabularnewline
\cline{1-9} 
Lin & 1.04 & 0.71 & 1.51 &  &  & 0.54 & 0.37 & 0.81 & \tabularnewline
Poly & 1.53 & 1.01 & 2.34 &  &  & 0.95 & 0.70 & 1.31 & \tabularnewline
\cline{1-9} 
 & \multicolumn{8}{c}{$g_{0}$ is Linear FewLarge $K=50$} & \tabularnewline
\cline{1-9} 
Lin & 3.53 & 1.78 & 7.78 &  &  & 2.72 & 1.96 & 3.74 & \tabularnewline
Poly & 4.76 & 2.87 & 8.29 &  &  & 4.15 & 3.26 & 5.60 & \tabularnewline
\cline{1-9} 
 & \multicolumn{8}{c}{$g_{0}$ is Linear ManySmall $K=10$} & \tabularnewline
\cline{1-9} 
Lin & 2.53 & 1.70 & 3.76 &  &  & 0.95 & 0.59 & 1.38 & \tabularnewline
Poly & 5.22 & 3.57 & 7.01 &  &  & 1.77 & 1.27 & 2.52 & \tabularnewline
\cline{1-9} 
 & \multicolumn{8}{c}{$g_{0}$ is Linear ManySmall $K=50$} & \tabularnewline
\cline{1-9} 
Lin & 19.70 & 13.32 & 28.09 &  &  & 4.66 & 3.13 & 6.96 & \tabularnewline
Poly & 20.49 & 14.23 & 28.27 &  &  & 6.22 & 4.28 & 8.88 & \tabularnewline
\cline{1-9} 
 & \multicolumn{8}{c}{$g_{0}$ is Convex FewLarge $K=10$} & \tabularnewline
\cline{1-9} 
Lin & 45.96 & 37.24 & 57.86 &  &  & 36.94 & 29.25 & 48.42 & \tabularnewline
Poly & 5.82 & 4.38 & 8.16 &  &  & 3.26 & 2.40 & 4.70 & \tabularnewline
\cline{1-9} 
 & \multicolumn{8}{c}{$g_{0}$ is Convex FewLarge $K=50$} & \tabularnewline
\cline{1-9} 
Lin & 72.03 & 56.68 & 91.19 &  &  & 44.44 & 36.51 & 55.90 & \tabularnewline
Poly & 26.40 & 20.22 & 36.27 &  &  & 14.58 & 11.59 & 20.43 & \tabularnewline
\cline{1-9} 
 & \multicolumn{8}{c}{$g_{0}$ is Convex ManySmall $K=10$} & \tabularnewline
\cline{1-9} 
Lin & 33.96 & 26.78 & 46.12 &  &  & 21.46 & 17.07 & 28.86 & \tabularnewline
Poly & 14.35 & 10.79 & 19.01 &  &  & 5.43 & 3.93 & 7.84 & \tabularnewline
\cline{1-9} 
 & \multicolumn{8}{c}{$g_{0}$ is Convex ManySmall $K=50$} & \tabularnewline
\cline{1-9} 
Lin & 52.55 & 42.57 & 68.41 &  &  & 32.56 & 24.28 & 42.81 & \tabularnewline
Poly & 42.55 & 34.79 & 50.73 &  &  & 23.48 & 17.94 & 29.86 & \tabularnewline
 &  &  &  &  &  &  &  &  & \tabularnewline
\end{longtable}

\section{Concluding Remarks\label{Section_Conclusion}}

This paper introduced a general framework for estimation of high dimensional
point processes, where the focus is on forecasting. The estimation
methodology is feasible using a greedy algorithm. The rates of consistency
in the case of many additive components are optimal. A set of examples
for the applicability of different estimation procedures and their
convergence rates are derived as corollaries of the main result. This
asymptotic analysis differs from the one where only a few variables
are active, which is usually addressed in the high dimensional literature.
In finance, because of very low signal to noise ratio, it is often
found that most of the variables are cross-sectionally correlated
but weak predictors. In consequence, none dominates. Hence the asymptotic
analysis carried out here is in this vein. The empirical study of
the prediction of buy and sell trade arrivals for futures on the New
Zealand dollar seems to confirm that using a small subset of the variables
might be suboptimal. Hence, it pays to use many variables as long
as they are properly aggregated.

More inferential procedures need to be devised in the case of high
dimensional model estimation. In finance, many applications require
an assessment of model performance out of sample. In high frequency,
the size of the dataset is large and the estimation procedures need
to be computationally feasible. This paper provides some solutions
in this direction. For very large sample sizes, one may need to give
up the use of the likelihood and work with approximations. In this
case, the intensity density could be directly modelled as an additive
model, and the likelihood replaced with a square loss contrast estimator
(e.g., Gaiffas and Guilloux, 2012). Applications in this vein will
be the subject of future research.

\part*{Supplementary Material to ``Estimation for the Prediction of Point
Processes with Many Covariates'' by Alessio Sancetta}

\setcounter{figure}{0} \renewcommand{\thefigure}{A.\arabic{figure}}

\setcounter{equation}{0} \renewcommand{\theequation}{A.\arabic{equation}} 

\setcounter{page}{1} 

\setcounter{section}{0} \renewcommand{\thesection}{A.\arabic{section}}

\section{Proofs of Results \label{Section_ProofsAppendix}}

The notation is collected in the next subsection so that the reader
can refer to it when needed. 

\subsection{Preliminary Lemmas and Notation}

Write $\mathcal{L}_{0}:=\mathcal{L}\left(B_{0},\Theta,\mathcal{W}\right)$,
$\bar{\mathcal{L}}:=\mathcal{L}\left(\bar{B},\Theta,\mathcal{W}\right)$
and $\mathcal{L}:=\mathcal{L}\left(B,\Theta,\mathcal{W}\right)$ for
arbitrary, but fixed $B$. By Condition \ref{Condition_parameterSpaceRestrictions},
the envelope function of $\bar{\mathcal{L}}$, is 
\begin{equation}
\sup_{g\in\bar{\mathcal{L}}}\sup_{z\in\mathbb{R}}\left|g\left(z\right)\right|\leq\bar{B}\bar{\theta}/\underline{w}=:\bar{g}.\label{EQ_largestFunction}
\end{equation}
From the main text, recall that $\bar{B}_{w}:=\bar{B}/\underline{w}$.
Throughout, to keep notation simpler, suppose that $K>1$. 

To ease notation, write $\Lambda\left(t\right)$ for $\int_{0}^{t}d\Lambda\left(s\right)=\int_{0}^{t}\lambda\left(X\left(s\right)\right)ds$,
$\int_{0}^{t}e^{g}d\mu$ for $\int_{0}^{t}e^{g\left(X\left(s\right)\right)}ds$
and similarly for $\int_{0}^{t}gdN$, $\int_{0}^{t}gd\Lambda$ $\int_{0}^{t}gd\mu$,
etc., where $\mu$ is the Lebesgue measure. Hence, arguments $X\left(t\right)$
and $t$ are dropped, but this should cause no confusion: all integrals
here are w.r.t. $dN\left(t\right)$, $d\mu\left(t\right)$ etc., and
the argument of all the functions is $X\left(t\right)$. Also, $\lambda\left(X\left(s\right)\right)=e^{g_{0}\left(X\left(s\right)\right)}$,
where $\bar{g}_{0}:=\left|g_{0}\right|_{\infty}$. With no loss of
generality, to keep notation simple, also suppose that $\left|g_{B_{0}}\right|_{\infty}\leq\bar{g}_{0}$
(if this were not the case, we can just redefine $\bar{g}_{0}$ to
be an upper bound for the uniform norms of $g_{0}$ and $g_{B_{0}}$(recall
the definition of $B_{0}$ in (\ref{EQ_B0Definition})). It then follows
from (\ref{EQ_B0Definition}) that $\sup_{B>0}\left|g_{B}\right|_{\infty}\leq\bar{g}_{0}$
because $g_{B}$ is the best uniform approximation for $g_{0}$ in
$\mathcal{L}\left(B\right)$, and for $B\geq B_{0}$, (\ref{EQ_B0Definition})
implies $g_{B}=g_{B_{0}}$. These facts will be used freely in the
proofs without further mention. Define the following random Hellinger
metric $d_{T}\left(g,g_{0}\right)=\sqrt{\frac{1}{2}\int_{0}^{T}\left(e^{g/2}-e^{g_{0}/2}\right)^{2}d\mu}$.
Sometimes, it will be useful to consider the identity $d_{T}^{2}\left(g,0\right)=\frac{1}{2}\int_{0}^{T}\left|e^{g/2}-1\right|^{2}d\mu$. 

\begin{lemma}\label{Lemma_d_T_bounds}Suppose that $f,f'$ are functions
on $\mathbb{R}^{K}$. Then, 
\begin{equation}
\frac{1}{8}\int_{0}^{T}\left(f-f'\right)^{2}e^{f'}d\mu\leq d_{T}^{2}\left(f,f'\right).\label{EQ_d_TEq2}
\end{equation}
\end{lemma}

\begin{proof}Multiplying and dividing by $e^{f'}$, 
\begin{equation}
d_{T}^{2}\left(f,f'\right)=\frac{1}{2}\int_{0}^{T}e^{f'}\left(e^{\left(f-f'\right)/2}-1\right)^{2}d\mu.\label{EQ_dTFirstBound}
\end{equation}
Expand the square in the above display 
\[
\left(e^{\left(f-f'\right)/2}-1\right)^{2}=e^{\left(f-f'\right)}-2e^{\left(f-f'\right)/2}+1.
\]
By Taylor expansion of the two exponentials, the above is equal to
\[
\sum_{j=0}^{\infty}\frac{\left(f-f'\right)^{j}}{j!}-2\sum_{j=0}^{\infty}\frac{\left(f-f'\right)^{j}}{j!}\left(\frac{1}{2}\right)^{j}+1=\sum_{j=2}^{\infty}\frac{\left(f-f'\right)^{j}}{j!}\left(1-\frac{1}{2^{j-1}}\right)\geq\frac{\left(f-f'\right)^{2}}{4}.
\]
Inserting in (\ref{EQ_dTFirstBound}) deduce (\ref{EQ_d_TEq2}). \end{proof}

\begin{lemma}\label{Lemma_predictableLR}Suppose that $\left|g_{B_{0}}\right|_{\infty}\leq\bar{g}_{0}$.
Then, 
\[
0\leq\int_{0}^{T}\left[\left(g_{0}-g_{B_{0}}\right)d\Lambda-\left(e^{g_{0}}-e^{g_{B_{0}}}\right)d\mu\right]\leq\frac{1}{2}e^{2\bar{g}_{0}}\int_{0}^{T}\left(g_{0}-g\right)^{2}d\Lambda.
\]

\end{lemma}

\begin{proof}By definition of $d\Lambda=e^{g_{0}}d\mu$,
\begin{eqnarray}
\int_{0}^{T}\left[\left(g_{0}-g\right)d\Lambda-\left(e^{g_{0}}-e^{g}\right)d\mu\right] & = & \int_{0}^{T}\left[\left(g_{0}-g\right)e^{g_{0}}-\left(e^{g_{0}}-e^{g}\right)\right]d\mu\nonumber \\
 & = & \int_{0}^{T}\left[\left(g_{0}-g\right)+e^{-\left(g_{0}-g\right)}-1\right]e^{g_{0}}d\mu.\label{EQ_predictableLR}
\end{eqnarray}
For any fixed real $x$, by Taylor series with remainder, for some
$x_{*}$ in the convex hull of $\left\{ 0,x\right\} $, 
\begin{equation}
e^{-x}-1+x=\frac{x^{2}}{2}e^{-x_{*}}.\label{EQ_logLikL2Inequality}
\end{equation}
Apply this equality to $x=g_{0}-g$ and insert it in the square brackets
on the r.h.s. of (\ref{EQ_predictableLR}) to deduce the upper bound
in the lemma because $\left|g_{0}-g_{B_{0}}\right|_{\infty}\leq2\bar{g}_{0}$.
For any $x>0$, the following inequality holds: 
\begin{equation}
0\leq\left(x-\ln x-1\right)\label{EQ_basicLogInequality}
\end{equation}
with equality only if $x=1$. Apply this inequality to $x=\exp\left\{ -\left(g_{0}-g_{B_{0}}\right)\right\} $
and insert it in the square brackets on the r.h.s. of (\ref{EQ_predictableLR})
to deduce the lower bound in the lemma. \end{proof}

\subsection{Solution of the Population Likelihood\label{Section_proofPopulationLikelihood}}

For simplicity, as in Condition \ref{Condition_stochasticRestrictions}
suppose $T_{0}=0$. Then, by Lemma 2 in Ogata (1978), 
\[
L\left(g\right)=\lim_{T}\frac{L_{T}\left(g\right)}{T}=\lim_{T}\frac{1}{T}\int_{0}^{T}\left(gdN-e^{g}d\mu\right)=P\left(ge^{g_{0}}-e^{g}\right)
\]
almost surely, where $L_{T}$ is the log-likelihood at time $T$ (e.g.,
Ogata, 1978, eq.1.3). Taking first derivatives, the first order condition
is $P\left(he^{g_{0}}-he^{g}\right)=0$ for any $h\in\bar{\mathcal{L}}$.
Hence, if $g=g_{0}$, the condition is satisfied. To check uniqueness,
verify that the second order condition for concavity, i.e., $-Ph^{2}e^{g}<0$
holds for any $h\neq0$. Using the lower bound $e^{-\bar{g}}\leq e^{g}$,
deduce that $-Ph^{2}e^{g}\leq-e^{-\bar{g}}Ph^{2}<0$ holds for any
$h\neq0$ $P$-almost everywhere. Given that $-L\left(g\right)$ is
convex and $\bar{\mathcal{L}}$ is convex and closed, the maximizer
of $L\left(g\right)$ is unique.

\subsection{Proof of Theorem \ref{Theorem_consistency}}

The result is derived for the Hellinger distance $d_{T}$ rather than
the norm $\left|\cdot\right|_{\lambda,T}$. 

Define $C_{T}^{2}:=C^{2}\times T\max\left\{ r_{T}^{-2},2e^{3\bar{g}_{0}}\left|g_{0}-g_{\bar{B}}\right|_{\infty}^{2}\right\} $
and the martingale $M=N-\Lambda$ ($\Lambda$ in (\ref{EQ_intensityRepresentation})
is the compensator of $N$). Here, $r_{T}$ is a nondecreasing sequence
which will be defined in due course. With the present notation, the
last display in the proof of lemma 4.1 in van de Geer (1995) states
that 
\begin{equation}
\frac{1}{2}\int_{0}^{T}\left(g-g_{0}\right)dM\geq d_{T}^{2}\left(g,g_{0}\right)+\frac{1}{2}L_{T}\left(g,g_{0}\right),\label{EQ_vanDeGeerHellingerIne}
\end{equation}
where $L_{T}\left(g,g_{0}\right):=L_{T}\left(g\right)-L_{T}\left(g_{0}\right)$
for any $g$, so also for $g=g_{T}$. (The above display is only valid
when $g_{0}$ is the true function, but it is not required that $g_{0}\in\mathcal{L}\left(B\right)$
for some $B$.) By Condition \ref{Condition_estimatorRestrictions},
and the inequality $L_{T}\left(g_{T},g_{\bar{B}}\right)\geq L_{T}\left(g_{T}\right)-\sup_{g\in\bar{\mathcal{L}}}L_{T}\left(g\right)$,
deduce that 
\begin{equation}
L_{T}\left(g_{T},g_{0}\right)=L_{T}\left(g_{T},g_{\bar{B}}\right)+L_{T}\left(g_{\bar{B}},g_{0}\right)\geq-\left(C_{T}^{2}/2\right)+L_{T}\left(g_{\bar{B}},g_{0}\right)\label{EQ_logLikLowerBound}
\end{equation}
 choosing $C$ large enough, in the definition of $C_{T}$. Hence,
inserting (\ref{EQ_logLikLowerBound}) in (\ref{EQ_vanDeGeerHellingerIne}),
deduce that 
\begin{eqnarray}
 &  & \Pr\left(d_{T}\left(g_{T},g_{0}\right)>C_{T}\right)\nonumber \\
 & \leq & \Pr\bigg(\frac{1}{2}\left[\int_{0}^{T}\left(g-g_{0}\right)dM-L_{T}\left(g_{\bar{B}},g_{0}\right)\right]\geq d_{T}^{2}\left(g,g_{0}\right)-\frac{C_{T}^{2}}{4}\label{EQ_d_TBoundL_T}\\
 &  & \text{ and }d_{T}^{2}\left(g,g_{0}\right)>C_{T}^{2}\text{ for some }g\in\bar{\mathcal{L}}\bigg)\nonumber 
\end{eqnarray}
To bound the term in the square bracket, add and subtract $\int_{0}^{T}g_{\bar{B}}dM$
and note that $L_{T}\left(g_{\bar{B}},g_{0}\right)$ can be written
as $\int_{0}^{T}\left[\left(g_{\bar{B}}-g_{0}\right)dM+\left(g_{\bar{B}}-g_{0}\right)d\Lambda-\left(e^{g_{\bar{B}}}-e^{g_{0}}\right)d\mu\right]$.
This implies that 
\begin{eqnarray*}
\int_{0}^{T}\left(g-g_{0}\right)dM-L_{T}\left(g_{\bar{B}},g_{0}\right) & = & \int_{0}^{T}\left[\left(g-g_{\bar{B}}\right)+\left(g_{\bar{B}}-g_{0}\right)\right]dM\\
 &  & -\int_{0}^{T}\left[\left(g_{\bar{B}}-g_{0}\right)dM+\left(g_{\bar{B}}-g_{0}\right)d\Lambda-\left(e^{g_{\bar{B}}}-e^{g_{0}}\right)d\mu\right]\\
 & = & \int_{0}^{T}\left(g-g_{\bar{B}}\right)dM+\int_{0}^{T}\left[\left(g_{0}-g_{\bar{B}}\right)d\Lambda-\left(e^{g_{0}}-e^{g_{\bar{B}}}\right)d\mu\right]\\
 & \leq & \int_{0}^{T}\left(g-g_{\bar{B}}\right)dM+\frac{1}{2}e^{2\bar{g}_{0}}\int_{0}^{T}\left(g_{0}-g_{\bar{B}}\right)^{2}d\Lambda
\end{eqnarray*}
using Lemma \ref{Lemma_predictableLR} in the inequality. From the
above calculations, and the fact that $\int_{0}^{T}\left(g_{0}-g_{\bar{B}}\right)^{2}d\Lambda\leq Te^{\bar{g}_{0}}\left|g_{0}-g_{\bar{B}}\right|_{\infty}^{2}$,
deduce that (\ref{EQ_d_TBoundL_T}) is less than 
\begin{eqnarray*}
 &  & \Pr\bigg(\frac{1}{2}\int_{0}^{T}\left(g-g_{\bar{B}}\right)dM\geq d_{T}^{2}\left(g,g_{0}\right)-\frac{C_{T}^{2}}{4}-\frac{1}{2}Te^{3\bar{g}_{0}}\left|g_{\bar{B}}-g_{0}\right|_{\infty}^{2}\\
 &  & \text{ and }d_{T}^{2}\left(g,g_{0}\right)>C_{T}^{2}\text{ for some }g\in\bar{\mathcal{L}}\bigg)\\
 & \leq & \Pr\left(\frac{1}{2}\int_{0}^{T}\left(g-g_{\bar{B}}\right)dM\geq d_{T}^{2}\left(g,g_{0}\right)-\frac{C_{T}^{2}}{2}\text{ and }d_{T}^{2}\left(g,g_{0}\right)>C_{T}^{2}\text{ for some }g\in\bar{\mathcal{L}}\right),
\end{eqnarray*}
using the definition of $C_{T}$. The above is bounded by $\Pr\left(\sup_{g\in\bar{\mathcal{L}}}\int_{0}^{T}\left(g-g_{\bar{B}}\right)dM\geq C_{T}^{2}\right)$,
which is further bounded by 
\[
\frac{1}{C_{T}^{2}}\mathbb{E}\left|\sup_{g\in\bar{\mathcal{L}}}\int_{0}^{T}\left(g-g_{\bar{B}}\right)dM\right|\leq\frac{2}{C_{T}^{2}}\mathbb{E}\left|\sup_{g\in\bar{\mathcal{L}}}\int_{0}^{T}gdM\right|
\]
using Markov inequality and then the triangle inequality because $g_{\bar{B}}\in\bar{\mathcal{L}}$.
Write $g=\sum_{\theta}b_{\theta}\theta$. Note that 
\[
\sup_{g\in\bar{\mathcal{L}}}\left|\int_{0}^{T}gdM\right|=\sup_{b_{\theta},\theta\in\Theta}\left|\int_{0}^{T}\left(\sum_{\theta}b_{\theta}\theta\right)dM\right|\leq\bar{B}_{w}\sup_{\theta\in\Theta}\left|\int_{0}^{T}\theta dM\right|
\]
where the supremum runs over all the $b_{\theta}$'s such that $\sum_{\theta}\left|b_{\theta}\right|\leq\bar{B}_{w}$.
According to these calculations, to bound (\ref{EQ_d_TBoundL_T})
it is sufficient to bound 
\begin{equation}
\frac{2\bar{B}_{w}}{C_{T}^{2}}\mathbb{E}\sup_{\theta\in\Theta}\left|\int_{0}^{T}\theta dM\right|.\label{EQ_uniformExpectationThetaMartingale}
\end{equation}
Let $\left\{ \Pi_{l}\left(\epsilon\right):v=1,2,..,N_{\Pi}\right\} $
be a partition of $\Theta$ into $N_{\Pi}\left(\epsilon\right)$ elements
such that $\sup_{\theta,\theta'\in\Pi_{l}\left(\epsilon\right)}\left|\theta-\theta'\right|\leq\epsilon$.
By Condition \ref{Condition_parameterSpaceRestrictions}, one can
construct such partition with $N_{\Pi}\left(\epsilon\right)\lesssim N\left(\epsilon,\Theta\right)$
and such that 
\begin{equation}
\sup_{\theta,\theta'\in\Pi_{l}\left(\epsilon\right)}\left|\theta-\theta'\right|_{\infty}\leq\left|\theta_{U,l}-\theta_{L,l}\right|_{\infty}\label{EQ_partitionBracket}
\end{equation}
where $\left[\theta_{L,l},\theta_{U,l}\right]$ is an $\epsilon$-bracket
for the functions in $\Pi_{l}$, under the uniform norm. It follows
that $N_{\Pi}\left(2\bar{\theta}\right)=1$ because the diameter of
$\Theta$ under the uniform norm is bounded by $2\bar{\theta}$. To
bound (\ref{EQ_uniformExpectationThetaMartingale}), use the following
maximal inequality from Nishiyama (1998, Theorem 2.2.3), which is
specialized to the present framework.

\begin{lemma}\label{Lemma_nishiyamaMaximalIneq} Under Condition
\ref{Condition_parameterSpaceRestrictions}), 
\begin{equation}
\mathbb{E}\max_{t\in\left[0,T\right]}\max_{\theta\in\Theta}\left|\int_{0}^{t}\theta dM\right|\lesssim C_{1,T}\int_{0}^{2\bar{\theta}}\sqrt{\ln\left(1+N_{\Pi}\left(\epsilon\right)\right)}d\epsilon+\frac{C_{2,T}}{\bar{\theta}C_{1,T}}\label{EQ_NishiyamaMaxInequ}
\end{equation}
for any $C_{2,T}\geq\int_{0}^{T}\bar{\theta}^{2}d\Lambda$, and $C_{1,T}\geq\left|\Theta\right|_{\Pi,T}$,
where 
\[
\left|\Theta\right|_{\Pi,T}:=\sup_{\epsilon\in(0,\bar{\theta}]}\max_{l\leq N_{\Pi}\left(\epsilon\right)}\frac{\sqrt{\int_{0}^{T}\left(\sup_{\theta,\theta'\in\Pi_{l}\left(\epsilon\right)}\left|\theta-\theta'\right|\right)^{2}d\Lambda}}{\epsilon}.
\]
\end{lemma}

From the discussion around (\ref{EQ_partitionBracket}) replace $N_{\Pi}\left(\epsilon\right)$
with $N\left(\epsilon,\Theta\right)$. Application of Lemma \ref{Lemma_nishiyamaMaximalIneq}
essentially requires to find a bound for $C_{1,T}$ and $C_{2,T}$.
Given that $\lambda=d\Lambda/d\mu$ is bounded by $e^{\bar{g}_{0}}$,
from the discussion around (\ref{EQ_partitionBracket}), $\left|\Theta\right|_{\Pi,T}\leq\sqrt{e^{\bar{g}_{0}}T}$
and we set $C_{1T}=C_{1}\sqrt{e^{\bar{g}_{0}}T}$ for some $C_{1}$
to be chosen later. Also, deduce that we can choose $C_{2,T}=\bar{\theta}e^{\bar{g}_{0}}T$.
This implies that $C_{2,T}/\bar{\theta}C_{1,T}=\sqrt{e^{\bar{g}_{0}}T/C_{1}}$.
Hence, the first term on the r.h.s. of (\ref{EQ_NishiyamaMaxInequ})
is of no smaller order of magnitude than the second (i.e., not smaller
than a constant multiple of $T^{1/2}$). Hence, in what follows, we
can incorporate $C_{2,T}/\bar{\theta}C_{1,T}$ into it without further
mention. Hence, an application of Lemma \ref{Lemma_nishiyamaMaximalIneq}
bounds (\ref{EQ_uniformExpectationThetaMartingale}) by 
\begin{equation}
\frac{2\bar{B}_{w}}{C_{T}^{2}}\mathbb{E}\sup_{\theta\in\Theta}\left|\int_{0}^{T}\theta dM\right|\lesssim\frac{2\bar{B}_{w}\sqrt{e^{\bar{g}_{0}}T}}{C_{T}^{2}}\int_{0}^{2\bar{\theta}}\sqrt{\ln\left(1+N\left(\epsilon,\Theta\right)\right)}d\epsilon.\label{EQ_L1UniformBoundEntropyIntegral}
\end{equation}
Using the definition of $C_{T}$, and choosing $r_{T}^{2}\lesssim\left[e^{3\bar{g}_{0}}\left|g_{0}-g\right|_{\infty}^{2}\right]^{-1}$,
the above is a constant multiple of 
\[
r_{T}^{2}\frac{\bar{B}_{w}e^{\bar{g}_{0}/2}}{T^{1/2}}\int_{0}^{2\bar{\theta}}\sqrt{\ln\left(1+N\left(\epsilon,\Theta\right)\right)}d\epsilon
\]
which is required to be $O\left(1\right)$, as it is an upper bound
for (\ref{EQ_d_TBoundL_T}) . This implies 
\[
r_{T}^{2}\lesssim\frac{T^{1/2}}{\bar{B}_{w}e^{\bar{g}_{0}/2}\int_{0}^{2\bar{\theta}}\sqrt{\ln\left(1+N\left(\epsilon,\Theta\right)\right)}d\epsilon}.
\]
But, $r_{T}$ is also required not to go to zero, and in fact it is
supposed to diverge to infinity unless the approximation error is
nonvanishing. Therefore, the r.h.s. of the above display needs to
be bounded away from zero. 

To bound the entropy integral, recall that $\Theta=\bigcup_{k=1}^{K}\Theta_{k}$.
The bracketing number of a union of sets is bounded above by the sum
of the bracketing numbers of the individual sets. Hence, $N\left(\epsilon,\Theta\right)\leq\sum_{k=1}^{K}N\left(\epsilon,\Theta_{k}\right)$.
Using the inequality $\ln\left(1+xy\right)\leq\ln x+\ln\left(1+y\right)$
for real $x,y\geq1$, this implies that
\begin{eqnarray*}
\int_{0}^{2\bar{\theta}}\sqrt{\ln\left(1+N\left(\epsilon,\Theta_{k}\right)\right)}d\epsilon & \leq & \int_{0}^{2\bar{\theta}}\max_{k\leq K}\sqrt{\ln K+\ln\left(1+N\left(\epsilon,\Theta_{k}\right)\right)}d\epsilon\\
 & \leq & 2\bar{\theta}\sqrt{\ln K}+\max_{k\leq K}\int_{0}^{2\bar{\theta}}\sqrt{\ln\left(1+N\left(\epsilon,\Theta_{k}\right)\right)}d\epsilon.
\end{eqnarray*}
Also, given that $\bar{\theta}$ is bounded and the entropy above
is decreasing in $\epsilon$, the above display can be bounded by
a multiple of 
\begin{equation}
\sqrt{\ln K}+\max_{k\leq K}\int_{0}^{1}\sqrt{\ln\left(1+N\left(\epsilon,\Theta_{k}\right)\right)}d\epsilon.\label{EQ_complexityBoundThConsistency}
\end{equation}
Also, we can discard the terms that are bounded, i.e., $\bar{g}_{0}$
and $\bar{\theta}$, but kept so far just to highlight what their
contribution might be. Similarly, $\bar{B}_{w}$ can be replaced by
$\bar{B}$ because it enters the bound as a multiplicative constant.
These calculations imply that there is a sequence $r_{T}$ as in the
statement in the theorem such that for $C$ large enough, 
\[
\Pr\left(\frac{r_{T}^{2}}{T}d_{T}^{2}\left(g_{T},g_{0}\right)>C\right)\leq\frac{1}{C^{2}}.
\]
By the relation between $d_{T}^{2}\left(g_{T},g_{0}\right)/T$ and
$\left|g_{T}-g_{0}\right|_{\lambda,T}^{2}$ (see (\ref{EQ_d_TEq2})),
the theorem follows.

\subsection{Proof of Theorem \ref{Theorem_optimality}}

To ease notation, $T=T_{n}$. We adapt the calculations in the proof
of Theorem 2 in Tsybakov (2003). This requires an upper bound for
the Kullback-Leibler distance between two intensity densities, and
the construction of a suitable subset of $\mathcal{L}\left(1\right)$
(using the notation of our theorem). The result in Tsybakov (2003)
will then provide the necessary lower bound as stated in Theorem \ref{Theorem_optimality}. 

To this end, let $N^{\left(1\right)}$ and $N^{\left(2\right)}$ be
point processes with intensities $e^{g_{1}}$ and $e^{g_{2}}$ such
that $\left|g_{k}\right|_{\infty}\leq\bar{g}$, $k=1,2$. Let the
sigma algebra generated by the process $X=\left(X\left(t\right)\right)_{t\geq0}$
be denoted by $\mathcal{F}^{X}$. The Kullback-Leibler distance between
two intensity densities $e^{g_{1}}$ and $e^{g_{2}}$, restricted
to $0,T]$, and conditioning on $\mathcal{F}^{X}$ is 
\[
K\left(g_{1},g_{2}|\mathcal{F}^{X}\right)=\mathbb{E}_{X}\int_{0}^{T}\left(g_{1}-g_{2}\right)dN^{\left(1\right)}-\int_{0}^{T}\left(e^{g_{1}}-e^{g_{2}}\right)d\mu
\]
where $\mathbb{E}_{X}$ is the expectation conditional on $\mathcal{F}^{X}$.
The above follows noting that conditioning on $\mathcal{F}^{X}$,
durations are exponentially distributed with intensity density $\exp\left\{ g_{1}\left(X\left(t\right)\right)\right\} $).
Then, 
\[
K\left(g_{1},g_{1}|\mathcal{F}^{X}\right)=\int_{0}^{T}\left(g_{1}-g_{2}\right)e^{g_{1}}d\mu-\int_{0}^{T}\left(e^{g_{1}}-e^{g_{2}}\right)d\mu\leq\frac{e^{3\bar{g}}}{2}\int_{0}^{T}\left|g_{1}-g_{2}\right|^{2}d\mu
\]
using (\ref{EQ_logLikL2Inequality}) and the fact that $\left|g_{k}\right|_{\infty}\leq\bar{g}$,
$k=1,2$. This provides the necessary upper bound for the Kullback-Leibler
distance, to be used in the proof of Theorem 2 in Tsybakov (2003). 

Now, follow Bunea et al. (2007, p. 1693) with minor adjustments. For
each $k$, we shall construct a function, say $f_{k}$, in $\Theta_{k}$.
Let $A_{j}=\sum_{i=1}^{j}1\left\{ T_{i}-T_{i-1}\geq a\right\} $,
i.e. the number of durations greater than $a$ amongst the first $j$
durations. Throughout, $1\left\{ \cdot\right\} $ is the indicator
function . Clearly, $A_{n}\leq n$ with equality only if $a=0$. Define
\[
f_{k}\left(x\right)=\gamma\sum_{j=1}^{n}\phi_{k}\left(\frac{A_{j}}{A_{n}}\right)\frac{1\left\{ x_{k}=X_{k}\left(T_{j-1}\right)\right\} 1\left\{ T_{j}-T_{j-1}\geq a\right\} }{\sqrt{T_{j}-T_{j-1}}}
\]
where $\gamma>0$ is a constant to be chosen in due course, and $\left\{ \phi_{k}\left(s\right):k=1,2,...,K\right\} $
are bounded functions w.r.t. $s\in\left[0,1\right]$, and such that
$\frac{1}{A_{n}}\sum_{j=1}^{A_{n}}\phi_{k}\left(\frac{j}{A_{n}}\right)\phi_{l}\left(\frac{j}{A_{n}}\right)=\delta_{kl}$,
where $\delta_{kl}=1$ if $k=l$, zero otherwise (e.g., mutatis mutandis,
as in Bunea et al., 2007, p. 1693). The functions $f_{k}$'s are uniformly
bounded in absolute value by a constant multiple of $\gamma/\sqrt{a}$.
Hence $f_{k}\in\Theta_{k}$, for each $k$, choosing $\gamma$ small
enough. It follows that 
\begin{eqnarray*}
\int_{0}^{T}f_{k}\left(X\left(t\right)\right)f_{l}\left(X\left(t\right)\right)dt & = & \sum_{j=1}^{n}f_{k}\left(X\left(T_{j-1}\right)\right)f_{l}\left(X\left(T_{j-1}\right)\right)\left(T_{j}-T_{j-1}\right)\\
 & = & \gamma^{2}\sum_{j=1}^{A_{n}}\phi_{k}\left(\frac{j}{A_{n}}\right)\phi_{l}\left(\frac{j}{A_{n}}\right)=\gamma^{2}A_{n}\delta_{kl}.
\end{eqnarray*}
The first step follows because $X\left(t\right)$ is predictable and
only changes after a jump. The second step follows by the definition
of the $f_{k}$'s because, by continuity of the distribution of $X\left(0\right)$,
and stationarity, $\Pr\left(X\left(T_{i}\right)=X\left(T_{j}\right)\right)=0$
for $i\neq j$. Also, note that, unless $\left\{ T_{j}-T_{j-1}\geq a\right\} $
is true, the $j^{th}$ term in the definition of $f_{k}$ will be
zero.

Let $\mathcal{C}$ be the subset of $\mathcal{L}\left(1\right)$ which
consists of arbitrary convex combinations of $m\leq K/6$ of the $f_{k}$'s
with weight $1/m$ so that the weights sum to one. In consequence,
for any $g_{1},g_{2}\in\mathcal{C}$, 
\[
\int_{0}^{T}\left(g_{1}-g_{2}\right)^{2}d\mu\asymp A_{n}\gamma^{2}/m.
\]
Let $p_{a}:=\Pr\left(T_{j}-T_{j-1}\geq a\right)$. We claim that $\Pr\left(A_{n}<np_{a}/2\right)\rightarrow0$
exponentially fast. Hence, the r.h.s. of the above display is proportional
to $n\gamma^{2}/m$ with probability going to one. This claim will
be verified at the end of the proof. 

Now, by suitable choice of small $\gamma$, it is possible to follow
line by line the argument after eq. (10) in Tsybakov (2003, proof
of Theorem 2). This would give us a result for $\int_{0}^{T}\left(g_{T}-g_{0}\right)^{2}d\mu$
rather than $\int_{0}^{T}\left(g_{T}-g_{0}\right)^{2}\lambda d\mu$
and in terms of $n$ rather than $T=T_{n}$. To replace $n$ with
$T_{n}$ as in the statement of the theorem, note that $T_{n}/n$
converges almost surely to $\left(P\lambda\right)^{-1}$, which is
bounded. Finally, $\int_{0}^{T}\left(g_{T}-g_{0}\right)^{2}\lambda d\mu\gtrsim\int_{0}^{T}\left(g_{T}-g_{0}\right)^{2}d\mu$
by the conditions of the theorem. 

It remains to show that the claim on $A_{n}$ holds true. For any
positive decreasing function $h$ on the reals, the sets $\left\{ A_{n}<cn\right\} $
and $\left\{ h\left(A_{n}\right)>h\left(c\right)\right\} $ are the
same; here $c\in\left(0,1\right)$ is a constant to be chosen in due
course. Hence, by Markov inequality, $\Pr\left(A_{n}<cn\right)\leq\mathbb{E}h\left(n^{-1/2}A_{n}\right)/h\left(cn^{1/2}\right)$,
which implies the following lower bound, 
\[
\Pr\left(A_{n}\geq cn\right)\geq1-\frac{\mathbb{E}h\left(A_{n}/\sqrt{n}\right)}{h\left(c\sqrt{n}\right)}.
\]
It remains to show that the second term on the r.h.s. goes to zero.
To this end, let $h\left(s\right)=e^{-ts}$, for some fixed $t>0$.
For $p_{a}$ as previously defined in the proof, write 
\[
\frac{A_{n}}{\sqrt{n}}=\frac{1}{\sqrt{n}}\sum_{i=1}^{n}\left(1\left\{ T_{i}-T_{i-1}\geq a\right\} -p_{a}\right)+\sqrt{n}p_{a}.
\]
The first term on the r.h.s. is a root-n standardized sum of i.i.d.
centered Bernoulli random variables. Hence, it has a moment generating
function which is bounded (use the proof of the central limit theorem
for Bernoulli random variables). By this remark, 
\[
\frac{\mathbb{E}h\left(A_{n}/\sqrt{n}\right)}{h\left(c\sqrt{n}\right)}=\frac{\mathbb{E}\exp\left\{ -tA_{n}/\sqrt{n}\right\} }{\exp\left\{ -tc\sqrt{n}\right\} }\lesssim e^{-t\left(p_{a}-c\right)\sqrt{n}}.
\]
Choose $c=p_{a}/2$ to see that the r.h.s. goes to zero exponentially
fast for any $t>0$, as previously claimed. 

\subsection{Proof of Lemma \ref{Lemma_L_B_B0Approximation} and Corollaries\label{Section_proofOfLemma1Corollaries}}

\begin{proof}{[}Lemma \ref{Lemma_L_B_B0Approximation}{]} The proof
is a minor re-adaptation of the Lemma 4 in Sancetta (2015). Note that
if $B\geq B_{0}$, the lemma is clearly true because, in this case,
$\mathcal{L}_{0}\subseteq\mathcal{L}:=\mathcal{L}\left(B,\Theta,\mathcal{W}\right)$.
Hence, assume $B<B_{0}$ and w.n.l.g. $B=\rho B_{0}$ for $\rho\in\left(0,1\right)$.
Write 
\[
g_{0}=\sum_{\theta\in\Theta}b_{\theta}\theta=\sum_{\theta\in\Theta}\lambda_{0\theta}\bar{b}\theta
\]
where the $\lambda_{\theta}$'s are nonnegative and add to one, and
$\bar{b}=\sum_{\theta\in\Theta}\left|b_{\theta}\right|$. Note that
the constraint $\sum_{\theta\in\Theta}w_{\theta}\left|b_{\theta}\right|\leq B_{0}$
for functions in $\mathcal{L}_{0}$ implies $\bar{b}\leq B_{0}/\underline{w}$.
Define $g'\left(x\right)=\rho g_{0}\left(x\right)$ for $\rho$ such
that $B=\rho B_{0}$ so that $g'\in\mathcal{L}$. Using this choice
of $g'$, by standard inequalities,
\[
\left|g_{0}-g'\right|_{r}\leq\left|\sum_{\theta\in\Theta}\lambda_{\theta}\bar{b}\theta-\sum_{\theta\in\Theta}\lambda_{\theta}\rho\bar{b}\theta\right|_{r}\leq\left|\bar{b}\left(1-\rho\right)\right|\sum_{\theta\in\Theta}\lambda_{\theta}\left|\theta\right|_{r}\leq\bar{b}\left(1-\rho\right)\max_{\theta\in\Theta}\left|\theta\right|_{r}\leq\frac{\bar{\theta}_{r}}{\underline{w}}\left(B_{0}-B\right)
\]
using the definition of $\rho$. This proves the result, because for
$g'$ above, $\inf_{g\in\mathcal{L}}\left|g_{0}-g\right|_{r}\leq\left|g_{0}-g'\right|_{r}$.\end{proof}

\begin{proof} {[}Corollary \ref{Lemma_asymptoticXs}{]} We need to
show that $L_{T}\left(\tilde{g}{}_{T},g_{B}\right)\geq-\left(C_{T}^{2}/2\right)$
with $C_{T}$ as in the proof of Theorem \ref{Theorem_consistency}
and $r_{T}$ as in (\ref{EQ_r_T}), e.g., $C_{T}^{2}\gtrsim\bar{B}\sqrt{T\ln K}$.
To this end, recall that $\tilde{L}_{T}\left(g\right)=\int_{0}^{T}g\left(\tilde{X}\left(t\right)\right)dN\left(t\right)-\int_{0}^{T}\exp\left\{ g\left(\tilde{X}\left(t\right)\right)\right\} dt$,
which is the log-likelihood when we use $\tilde{X}$ instead of $X$.
Note that the counting process $N$ is still the same, whether we
use $X$ or $\tilde{X}$, as jumps are observable. By definition,
$\tilde{g}$ is the approximate maximizer of $\tilde{L}_{T}\left(g\right)$,
but not necessarily the maximizer of $L_{T}\left(g\right)$. It would
be enough to show that $L_{T}\left(\tilde{g}{}_{T},g_{B}\right)\gtrsim-C_{T}^{2}$
in probability, as by a re-definition of the constant in $C_{T}$,
the proof in Theorem \ref{Theorem_consistency} would go through.
Given these remarks, write 
\[
L_{T}\left(\tilde{g}{}_{T},g_{B}\right)\geq\tilde{L}_{T}\left(\tilde{g}{}_{T},g_{B}\right)-\left|L_{T}\left(\tilde{g}{}_{T},g_{B}\right)-\tilde{L}_{T}\left(\tilde{g}{}_{T},g_{B}\right)\right|.
\]
Using (\ref{EQ_surrogateLogLikApproxMax}) we have that $\tilde{L}_{T}\left(\tilde{g}{}_{T},g_{B}\right)\gtrsim-C_{T}^{2}$
as in (\ref{EQ_logLikLowerBound}). To bound the second term on the
r.h.s. of the above display, it is sufficient to bound a constant
multiple of 
\begin{eqnarray*}
 &  & \sup_{g\in\bar{\mathcal{L}}}\left|L_{T}\left(g\right)-\tilde{L}_{T}\left(g\right)\right|\\
 & = & \sup_{g\in\bar{\mathcal{L}}}\left|\int_{0}^{T}\left[g\left(X\left(t\right)\right)-g\left(\tilde{X}\left(t\right)\right)\right]dN\left(t\right)-\int_{0}^{T}\left[\exp\left\{ g\left(X\left(t\right)\right)\right\} -\exp\left\{ g\left(\tilde{X}\left(t\right)\right)\right\} \right]dt\right|\\
 & \leq & \sup_{g\in\bar{\mathcal{L}}}\left|\int_{0}^{T}\left[g\left(X\left(t\right)\right)-g\left(\tilde{X}\left(t\right)\right)\right]dN\left(t\right)\right|+\sup_{g\in\bar{\mathcal{L}}}\left|\int_{0}^{T}\left[\exp\left\{ g\left(X\left(t\right)\right)\right\} -\exp\left\{ g\left(\tilde{X}\left(t\right)\right)\right\} \right]dt\right|\\
 & =: & I+II.
\end{eqnarray*}
First, find a bound for $II$. By the mean value theorem in Banach
spaces, 
\begin{equation}
II\leq\sup_{g\in\bar{\mathcal{L}}}e^{\bar{g}}\int_{0}^{T}\left|g\left(X\left(t\right)\right)-g\left(\tilde{X}\left(t\right)\right)\right|dt.\label{EQ_stationaryNonStationaryL1Bound}
\end{equation}
 Now, 
\begin{eqnarray*}
\sup_{g\in\bar{\mathcal{L}}}\int_{0}^{T}\left|g\left(X\left(t\right)\right)-g\left(\tilde{X}\left(t\right)\right)\right|dt & \leq & \sup_{\left\{ b_{\theta}:\sum_{\theta\in\Theta}\left|b_{\theta}\right|\leq\bar{B}_{w}\right\} }\int_{0}^{T}\sum_{\theta\in\Theta}\left|b_{\theta}\right|\left|\theta\left(\tilde{X}\left(t\right)\right)-\theta\left(X\left(t\right)\right)\right|dt\\
 & \leq & \bar{B}_{w}\max_{\theta\in\Theta}\int_{0}^{T}\left|\theta\left(\tilde{X}\left(t\right)\right)-\theta\left(X\left(t\right)\right)\right|dt
\end{eqnarray*}
because the supremum over the simplex is achieved at one of its edges.
By the conditions of the lemma, the above display is $O_{p}\left(\bar{B}e^{-\bar{g}}\sqrt{T\ln K}\right)$.
Hence, deduce that (\ref{EQ_stationaryNonStationaryL1Bound}) is $O_{p}\left(\bar{B}\sqrt{T\ln K}\right)=O_{p}\left(C_{T}\right)$
(recall the notation in (\ref{EQ_largestFunction})). 

It remains to bound $I$. Adding and subtracting $\int_{0}^{T}\left[g\left(X\left(t\right)\right)-g\left(\tilde{X}\left(t\right)\right)\right]d\Lambda\left(t\right)$
, and using the triangle inequality, 
\[
I\leq\sup_{g\in\bar{\mathcal{L}}}\left|\int_{0}^{T}\left[g\left(X\left(t\right)\right)-g\left(\tilde{X}\left(t\right)\right)\right]dM\left(t\right)\right|+\sup_{g\in\bar{\mathcal{L}}}\left|\int_{0}^{T}\left[g\left(X\left(t\right)\right)-g\left(\tilde{X}\left(t\right)\right)\right]d\Lambda\left(t\right)\right|.
\]
The first term in the above display can be incorporated in the l.h.s.
of (\ref{EQ_vanDeGeerHellingerIne}), and bounded as in the proof
of Theorem \ref{Theorem_consistency}. To bound the second term on
the above display, by definition of $d\Lambda$, 
\[
\sup_{g\in\bar{\mathcal{L}}}\left|\int_{0}^{T}\left[g\left(X\left(t\right)\right)-g\left(\tilde{X}\left(t\right)\right)\right]\exp\left\{ g_{0}\left(X\left(t\right)\right)\right\} dt\right|\leq\sup_{g\in\bar{\mathcal{L}}}e^{\bar{g}_{0}}\int_{0}^{T}\left|g\left(X\left(t\right)\right)-g\left(\tilde{X}\left(t\right)\right)\right|dt.
\]
From the derived bound for $II$ deduce that the r.h.s. is $O_{p}\left(C_{T}^{2}\right)$.
This completes the proof of the first statement in the corollary,
as all the conditions of Theorem \ref{Theorem_consistency} are satisfied.
To show the last statement of the corollary, use the inequality $\left|g\left(X\left(t\right)\right)-g\left(\tilde{X}\left(t\right)\right)\right|^{2}\leq2\bar{g}\left|g\left(X\left(t\right)\right)-g\left(\tilde{X}\left(t\right)\right)\right|$
together with a trivial modification of the previous display. \end{proof}

\begin{proof}{[}Corollary \ref{Corollary_linear}{]} The approximation
error is zero by assumption. Given that $\Theta_{k}$ has one single
element, the entropy integral is trivially finite. Hence, (\ref{EQ_r_T})
simplifies as in the statement of the corollary.\end{proof}

\begin{proof}{[}Corollary \ref{Lemma_hawkesApproximation}{]} Define
the set 
\[
\mathcal{B}:=\left\{ \sup_{t>0}\left|\int_{0}^{t}\left(t-s\right)e^{-\underline{a}\left(t-s\right)}dN\left(s\right)\right|\leq\beta\right\} 
\]
for some $\beta<\infty$. In the proof of Theorem \ref{Theorem_consistency}
write 
\[
\Pr\left(d_{T}\left(g_{T},g_{0}\right)>C_{T},\right)\leq\Pr\left(d_{T}\left(g_{T},g_{0}\right)>C_{T},\text{ and }\mathcal{B}\right)+\Pr\left(\mathcal{B}^{c}\right)
\]
where $\mathcal{B}^{c}$ is the complement of $\mathcal{B}$. We shall
apply Corollary \ref{Lemma_asymptoticXs} to the first term on the
r.h.s., and then show that the last term in the above display is negligible.

At first, show that the process with intensity density $\lambda\left(t\right)=\exp\left\{ f_{a_{0}}\left(t\right)+g_{0}\left(X\left(t\right)\right)\right\} $
is stationary. To this end, we apply Theorem 2 in Brémaud and Massoulié
(1996). Using their notation, their nonlinear function $\phi\left(\cdot\right)$
in their eq.(1) is here defined as $\exp\left\{ f\left(\cdot\right)\right\} \exp\left\{ g_{0}\left(X\left(t\right)\right)\right\} $,
which is random, unlike their case. However, in the proof of their
Theorem 2, they only use the fact that $\left|\phi\left(y\right)-\phi\left(y'\right)\right|\leq\alpha\left|y-y'\right|$
for some finite constant $\alpha$ (see their eq.(23) and first display
on p.1580). This is the case here as well. To see this, recall the
definition of $f$ (see Section \ref{Section_hawkesProcess}), which
is bounded and Lipschitz. Then, 
\[
\left|\exp\left\{ f\left(y\right)\right\} \exp\left\{ g_{0}\left(X\left(t\right)\right)\right\} -\exp\left\{ f\left(y'\right)\right\} \exp\left\{ g_{0}\left(X\left(t\right)\right)\right\} \right|\le\exp\left\{ \bar{g}_{0}\right\} \left|f\left(y\right)-f\left(y'\right)\right|
\]
(recall $\bar{g}_{0}$ is the uniform norm of $g_{0}$). We also need
to note that $\exp\left\{ g_{0}\left(X\left(t\right)\right)\right\} $
is stationary, bounded and predictable. This ensures that the intensity
$\lambda\left(t\right)$ is bounded and predictable, which is required
in the lemmas used in Brémaud and Massoulié (1996). Hence Condition
\ref{Condition_stochasticRestrictions} is satisfied. 

To verify Condition \ref{Condition_parameterSpaceRestrictions}, we
verify that the entropy integral of the process $\tilde{f}_{a}$ is
finite, in a sense to be made clear below. We shall postpone this
to the end of the proof. 

Hence, mutatis mutandis, we now verify (\ref{eq_stationaryApproximation})
in Corollary \ref{Lemma_asymptoticXs}. To this end, we bound $c_{T}:=\mathbb{E}\max_{a\in\left[\underline{a},\bar{a}\right]}\int_{0}^{T}\left|f_{a}\left(t\right)-\tilde{f}_{a}\left(t\right)\right|dt$.
Corollary \ref{Lemma_asymptoticXs} requires $c_{T}$ to be $O\left(e^{-\bar{B}_{w}\bar{\theta}}\sqrt{T\ln K}\right)$.
By the Lipschitz condition and $a\in\left[\underline{a},\bar{a}\right]$,
\[
\int_{0}^{T}\left|f_{a}\left(t\right)-\tilde{f}_{a}\left(t\right)\right|dt\lesssim\int_{0}^{T}e^{-\underline{a}t}\left(\int_{\left(-\infty,0\right)}e^{\underline{a}s}dN\left(s\right)\right)dt.
\]
Using the fact that $\Lambda$ is the compensator of $N$, and that
$\Lambda$ has bounded density $\exp\left\{ f_{a_{0}}\left(t\right)+g_{0}\left(X\left(t\right)\right)\right\} $,
deduce that 
\begin{eqnarray*}
\mathbb{E}\max_{a\in\left[\underline{a},\bar{a}\right]}\int_{0}^{T}\left|f_{a}\left(t\right)-\tilde{f}_{a}\left(t\right)\right|dt & \leq & \mathbb{E}\left[\left(\int_{\left(-\infty,0\right)}e^{\underline{a}s}dN\left(s\right)\right)\left(\int_{0}^{T}e^{-\underline{a}t}dt\right)\right]\\
 & \lesssim & \frac{1}{\underline{a}}\mathbb{E}\int_{\left(-\infty,0\right)}e^{\underline{a}s}d\Lambda\left(s\right)\lesssim\frac{1}{\underline{a}^{2}}<\infty.
\end{eqnarray*}
This verifies (\ref{eq_stationaryApproximation}) in Corollary \ref{Lemma_asymptoticXs}. 

To verify Condition \ref{Condition_parameterSpaceRestrictions} for
$\tilde{f}_{a}$, we need an estimate of the entropy integral for
the family of stochastic processes $\mathcal{A}:=\left\{ \left(\tilde{f}_{a}\left(t\right)\right)_{t\geq0}:a\in\left[\underline{a},\bar{a}\right]\right\} $.
This means that we need to bound 
\begin{eqnarray*}
\sup_{t>0}\left|\tilde{f}_{a}\left(t\right)-\tilde{f}_{a'}\left(t\right)\right| & \lesssim & \sup_{t>0}\left|\int_{0}^{t}\left(e^{-a\left(t-s\right)}-e^{-a'\left(t-s\right)}\right)dN\left(s\right)\right|\\
 & \leq & \sup_{t>0}\left|\int_{0}^{t}\left(t-s\right)e^{-\underline{a}\left(t-s\right)}dN\left(s\right)\right|dt\left|a-a'\right|
\end{eqnarray*}
using a first order Taylor expansion, and the lower bound on $a,a'$.
On $\mathcal{B}$, the above is $\beta\left|a-a'\right|$. It is then
easy to see that the entropy integral is a constant multiple of $\beta^{1/2}$
because the uniform $\epsilon$-bracketing number of $\left[\underline{a}\beta,\bar{a}\beta\right]$
has size $\beta\left(\bar{a}-\underline{a}\right)/\epsilon$. 

In consequence, we can apply Corollary \ref{Lemma_asymptoticXs}.
Let $\beta=O\left(\ln T\right)$). There is no approximation error,
so that $r_{T}^{-2}$ ($r_{T}$ as in (\ref{EQ_r_T})) becomes as
in (\ref{EQ_hawkesBoundRate}). The term $\sqrt{\ln T}$, in the denominator
of (\ref{EQ_hawkesBoundRate}), is proportional to the entropy integral
of $\mathcal{A}$.

To conclude, we show that $\mathcal{B}^{c}$, the complement of $\mathcal{B}$,
is such that $\Pr\left(\mathcal{B}^{c}\right)\rightarrow0$ as $\beta\rightarrow\infty$.
By Markov inequality, 
\[
\Pr\left(\mathcal{B}^{c}\right)\leq\frac{\mathbb{E}\sup_{t>0}\left|\int_{0}^{t}\left(t-s\right)e^{-\underline{a}\left(t-s\right)}dN\left(s\right)\right|}{\beta}.
\]
Recalling that $M=N-\Lambda$, by the triangle inequality, the numerator
on the r.h.s. can be bounded by 
\[
\mathbb{E}\sup_{t>0}\left|\int_{0}^{t}\left(t-s\right)e^{-\underline{a}\left(t-s\right)}dM\left(s\right)\right|+\mathbb{E}\sup_{t>0}\left|\int_{0}^{t}\left(t-s\right)e^{-\underline{a}\left(t-s\right)}d\Lambda\left(s\right)\right|=:I+II.
\]
The first integral inside the square is a bounded predictable function
w.r.t. a martingale, and is a martingale. By the Burkholder-Davis-Gundy
inequality, 
\[
I^{2}\lesssim\int_{0}^{\infty}\left|\left(t-s\right)e^{-\underline{a}\left(t-s\right)}\right|^{2}d\Lambda\left(s\right)\leq e^{\bar{g}_{0}}\int_{0}^{\infty}\left|\left(t-s\right)e^{-\underline{a}\left(t-s\right)}\right|^{2}ds=O\left(1\right).
\]
By a similar argument $II=O\left(1\right)$. These bounds imply that
$\Pr\left(\mathcal{B}^{c}\right)\rightarrow0$. The last statement
in the corollary is deduced from the proof of Corollary \ref{Corollary_linear}.\end{proof}

\begin{proof}{[}Corollary \ref{Corollary_thresholdModel}{]} By Lemma
\ref{Lemma_L_B_B0Approximation}, the approximation error will be
zero as soon as $\bar{B}\geq B_{0}$, which will be eventually the
case as $\bar{B}\rightarrow\infty$ and $B_{0}$ is finite. By the
remarks in Section \ref{Section_thresholdModel} the entropy integral
is finite. Hence, the bound follows from (\ref{EQ_r_T}).\end{proof}

\begin{proof}{[}Corollary \ref{Corollary_trigonometricPoly}{]} By
Lemma \ref{Lemma_LInfinityApproximation} and \ref{EQ_approximationByUnivariate}
the approximation error is a constant multiple of $V^{-2\alpha}+\max\left\{ c_{\alpha}-\bar{B},0\right\} ^{2}$.
The univariate square uniform approximation rate $V^{-2\alpha}$ follows
by the remarks in Section \ref{Section_Example_l1ExpansionFourierHolder}.
Given that there are $V$ elements in each $\Theta_{k}$ the entropy
integral is a constant multiple of $\sqrt{\ln\left(1+V\right)}$.
Inserting in (\ref{EQ_r_T}), the bound is deduced as long as $V>1$.
In particular for $V\gtrsim\left(T/\ln T\right)^{1/\left(4\alpha\right)}$
the bound simplifies further.\end{proof}

\begin{proof}{[}Corollary \ref{Corollary_neuralNet}{]} The proof
is the same as for Corollary \ref{Corollary_trigonometricPoly}.\end{proof}

\begin{proof}{[}Corollary \ref{Corollary_monotoneBernstein}{]} As
stated in Section \ref{Section_MonotoneBernstein}, the approximation
rate of Bernstein polynomials under the squared uniform loss is a
constant multiple of $\alpha^{2}V^{-1}$. Hence, by Lemma \ref{Lemma_LInfinityApproximation}
and (\ref{EQ_approximationByUnivariate}), the approximation error
is a constant multiple of $\alpha^{2}V^{-1}+\max\left\{ B_{0}-\bar{B},0\right\} ^{2}$.
In consequence, as $\bar{B}\rightarrow\infty$, the approximation
error is eventually $O\left(\sqrt{\alpha/T}\right)$ when $V\gtrsim T^{1/2}\alpha^{3/2}$.
By the remarks in Section \ref{Section_MonotoneBernstein}, the entropy
integral is $\alpha^{1/2}$. Inserting in (\ref{EQ_r_T}) the bound
follows.\end{proof}

\subsection{Proof of Theorem \ref{Theorem_FWA}}

Define $h:=b\theta$, and let $t\in\left[0,1\right]$. Let 
\[
h_{m}:=\arg\sup_{h\in\bar{\mathcal{L}}}D_{T}\left(F_{m-1},h-F_{m-1}\right).
\]
By linearity, the maximum is obtained by a function $h=b\theta$ with
$\theta\in\Theta_{k}$ for some $k$ and $\left|b\right|\leq\bar{B}$.
Hence is sufficient to maximize the absolute value of $D_{T}$ w.r.t.
$\theta$ as the coefficient $b$ is not constrained in sign. Define,
\[
G\left(F_{m-1}\right):=D_{T}\left(F_{m-1},h_{m}-F_{m-1}\right),
\]
so that for any $g\in\bar{\mathcal{L}}$, 
\begin{equation}
L_{T}\left(g\right)-L_{T}\left(F_{m-1}\right)\leq G\left(F_{m-1}\right)\label{EQ_gapEquation}
\end{equation}
by concavity. For $m\geq0$, define $\bar{\rho}_{m}=2/\left(m+2\right)$.
By concavity, again, 
\[
L_{T}\left(F_{m}\right)=\max_{\rho\in\left[0,1\right]}L_{T}\left(F_{m-1}+\rho\left(h-F_{m-1}\right)\right)\geq L_{T}\left(F_{m-1}\right)+D_{T}\left(F_{m-1},h-F_{m-1}\right)\bar{\rho}_{m}+\frac{\bar{C}}{2}\bar{\rho}_{m}^{2}
\]
where 
\[
\bar{C}:=\min_{h,g\in\bar{\mathcal{L}},t\in\left[0,1\right]}\frac{2}{t^{2}}\left[L_{T}\left(g+t\left(h-g\right)\right)-L_{T}\left(g\right)-D_{T}\left(g,t\left(h-g\right)\right)\right]<0.
\]
The above two displays together with (\ref{EQ_gapEquation}), imply
\begin{eqnarray}
L_{T}\left(F_{m}\right)-L_{T}\left(g\right) & \geq & L_{T}\left(F_{m-1}\right)-L_{T}\left(g\right)+\bar{\rho}_{m}G\left(F_{m-1}\right)+\frac{\bar{C}}{2}\bar{\rho}_{m}^{2}\nonumber \\
 & \geq & L_{T}\left(F_{m-1}\right)-L_{T}\left(g\right)+\bar{\rho}_{m}\left(L_{T}\left(g\right)-L_{T}\left(F_{m-1}\right)\right)+\frac{\bar{C}}{2}\bar{\rho}_{m}^{2}\nonumber \\
 & = & \left(1-\bar{\rho}_{m}\right)\left(L_{T}\left(F_{m-1}\right)-L_{T}\left(g\right)\right)+\frac{\bar{C}}{2}\bar{\rho}_{m}^{2}\nonumber \\
 & \geq & \frac{2\bar{C}}{m+2}\label{EQ_FWARecursionBound}
\end{eqnarray}
for the given choice of $\bar{\rho}_{m}$ (mutatis mutandis, as in
the proof of Theorem 1 in Jaggi (2013)). It remains to bound $\bar{C}$.
By Taylor's expansion in Banach spaces, 
\[
L_{T}\left(g+t\left(h-g\right)\right)=L_{T}\left(g\right)+D_{T}\left(g,t\left(h-g\right)\right)+\frac{1}{2}H_{T}\left(g_{*},t^{2}\left(h-g\right)^{2}\right),
\]
for $g_{*}=t_{*}g+\left(1-t_{*}\right)h$, and some $t_{*}\in\left[0,1\right]$,
where 
\[
H_{T}\left(g,t^{2}\left(h-g\right)\right)=-\int_{0}^{T}t^{2}\left(h-g\right)^{2}e^{g}ds.
\]
This means that 
\[
\bar{C}\geq\min_{h,g\in\bar{\mathcal{L}},t\in\left[0,1\right]}\frac{2}{t^{2}}\left[-\frac{1}{2}\int_{0}^{T}t^{2}\left(h\left(X\left(s\right)\right)-g\left(X\left(s\right)\right)\right)^{2}e^{\bar{g}}ds\right]\geq-4Te^{\bar{g}}\bar{g}^{2}\geq-4Te^{\bar{B}\bar{\theta}/\underline{w}}\left(\bar{B}\bar{\theta}/\underline{w}\right)^{2}
\]
using (\ref{EQ_largestFunction}). Substituting in (\ref{EQ_FWARecursionBound})
gives the result.

\subsection{Proof of Proposition \ref{Proposition_outOfSampleTest}}

Let $M:=N-\Lambda$ and $h_{t}:=g_{t}-g_{t}'$. To ease notation,
suppose for the moment that $S$ is an integer. Then, under the conditions
of the proposition (the null hypothesis),
\[
L_{S}\left(g,g'\right)=\sum_{s=1}^{S}\int_{s-1}^{s}h_{t}\left(X\left(t\right)\right)dM\left(t\right)=\sum_{s=1}^{S}Y_{s}.
\]
Then, $\left\{ Y_{s}:s=1,2,...\right\} $ is a sequence of martingale
differences. This follows from the law of iterated expectations and
the fact that $h_{t}$ is a predictable process. Denote the expectation
conditioning on $\left\{ Y_{i}:i\leq s\right\} $ by $\mathbb{E}_{s}$.
The result will follow by an application of Theorem 2.3 in McLeish
(1974). To this end, it is sufficient to show that (i.) $\mathbb{E}\left|\frac{1}{S}\sum_{s=1}^{S}Y_{s}^{2}\right|\rightarrow\sigma^{2}$,
(ii.) $\lim_{S\rightarrow\infty}\mathbb{E}\max_{s\leq S}Y_{s}^{2}/S<\infty$
and (iii.) $\max_{s\leq S}\left|Y_{s}/\sqrt{S}\right|\rightarrow0$
in probability. Note that 
\begin{equation}
\mathbb{E}\left|\frac{1}{S}\sum_{s=1}^{S}Y_{s}^{2}\right|=\mathbb{E}\left|\frac{1}{S}\sum_{s=1}^{S}\mathbb{E}_{s-1}Y_{s}^{2}\right|\label{EQ_predictableVarProposition}
\end{equation}
using iterated expectations and the fact the elements in the sum are
positive. Note that 
\begin{eqnarray*}
\mathbb{E}_{s-1}Y_{s}^{2} & = & \mathbb{E}_{s-1}\left[\int_{s-1}^{s}h_{t}\left(X\left(t\right)\right)dM\left(t\right)\right]^{2}\\
 & = & \mathbb{E}_{s-1}\left[\int_{s-1}^{s}h_{t}^{2}\left(X\left(t\right)\right)d\Lambda\left(t\right)\right]
\end{eqnarray*}
(e.g., Ogata, 1978, e.q. 2.1). Hence, 
\[
\frac{1}{S}\sum_{s=1}^{S}\mathbb{E}_{s-1}Y_{s}^{2}=\left[\frac{1}{S}\sum_{s=1}^{S}\mathbb{E}_{s-1}\int_{s-1}^{s}h_{t}^{2}\left(X\left(t\right)\right)d\Lambda\left(t\right)\right].
\]
By these remarks, (\ref{EQ_predictableVarProposition}) is equal to
\begin{eqnarray*}
\mathbb{E}\left|\frac{1}{S}\sum_{s=1}^{S}\mathbb{E}_{s-1}\int_{s-1}^{s}h_{t}^{2}\left(X\left(t\right)\right)d\Lambda\left(t\right)\right| & = & \frac{1}{S}\sum_{s=1}^{S}\mathbb{E}\int_{s-1}^{s}h_{t}^{2}\left(X\left(t\right)\right)d\Lambda\left(t\right)\\
 & = & \mathbb{E}\frac{1}{S}\int_{0}^{S}h_{t}^{2}\left(X\left(t\right)\right)d\Lambda\left(t\right),
\end{eqnarray*}
using the fact that the terms in the sum are positive. By the conditions
of the proposition 
\[
\sigma_{S}^{2}:=\frac{1}{S}\int_{0}^{S}h_{t}^{2}\left(X\left(t\right)\right)d\Lambda\left(t\right)\rightarrow\sigma^{2}>0
\]
in probability. The sequence $\left(\sigma_{S}^{2}\right)_{S\geq1}$
is uniformly bounded. In consequence, convergence in probability implies
convergence in $L_{1}$, i.e. $\mathbb{E}\sigma_{S}^{2}\rightarrow\sigma^{2}$.
This verifies the first condition (i.). Now, 
\[
\mathbb{E}\max_{s\leq S}\frac{Y_{s}^{2}}{S}\leq\frac{1}{S}\mathbb{E}\sum_{s=1}^{S}Y_{s}^{2}
\]
bounding the maximum by the sum. By the previous calculations deduce
that the above is bounded, which then verifies the second condition
(ii.). Finally, 
\begin{eqnarray*}
\max_{s\leq S}\left|Y_{s}\right|/\sqrt{S} & = & \frac{1}{\sqrt{S}}\max_{s\leq S}\left|\int_{s-1}^{s}h_{t}\left(X\left(t\right)\right)dM\left(t\right)\right|\\
 & \lesssim & \frac{1}{\sqrt{S}}\max_{s\leq S}\left|\int_{s-1}^{s}dN\left(t\right)\right|+\frac{1}{\sqrt{S}}\max_{s\leq S}\left|\int_{s-1}^{s}d\Lambda\left(t\right)\right|\\
 & = & \frac{1}{\sqrt{S}}\max_{s\leq S}\left[N\left(s\right)-N\left(s-1\right)\right]+\frac{1}{\sqrt{S}}\max_{s\leq S}\Lambda\left(\left[s-1,s\right]\right)
\end{eqnarray*}
where the inequality uses the fact that $h_{t}$ is bounded. The last
term on the r.h.s. is $O_{p}\left(S^{-1/2}\right)$. A counting process
$N$ is increasing with the intensity. Since $\lambda\left(X\left(s\right)\right)\leq e^{\bar{g_{0}}}$
uniformly in $s$, there is a counting process $N'$ with intensity
density $e^{\bar{g_{0}}}$ such $\Pr\left(N\left(s\right)>n\right)\leq\Pr\left(N'\left(s\right)>n\right)$.
In consequence, for any $s$, $\mathbb{E}\left[N\left(s\right)-N\left(s-1\right)\right]^{4}\leq\mathbb{E}\left[N'\left(s\right)-N'\left(s-1\right)\right]^{4}\leq C$
for some absolute constant $C$ that depends on $\bar{g}_{0}$ only.
The last inequality follows because $N'$ is Poisson with intensity
$e^{\bar{g_{0}}}$. By these remarks, 
\begin{eqnarray*}
\mathbb{E}\frac{1}{\sqrt{S}}\max_{s\leq S}\left[N\left(s\right)-N\left(s-1\right)\right] & \leq & \frac{1}{\sqrt{S}}\left(\mathbb{E}\max_{s\leq S}\left|N\left(s\right)-N\left(s-1\right)\right|^{4}\right)^{1/4}\\
 & \leq & \frac{1}{\sqrt{S}}\left(\sum_{s=1}^{S}\mathbb{E}\left|N\left(s\right)-N\left(s-1\right)\right|^{4}\right)^{1/4}
\end{eqnarray*}
bounding the maximum by the sum. Deduce that the above is $\left(C/S\right)^{1/4}=o\left(1\right)$.
This verifies the third condition (iii.) required for the application
of Theorem 2.3 in McLeish (1974). 

If $S$ is not an integer, write $\left\lfloor S\right\rfloor $ for
its integer part. Then, 
\[
\frac{1}{\sqrt{S}}L_{S}\left(g,g'\right)=\left(\frac{\left\lfloor S\right\rfloor }{S}\right)^{1/2}\frac{1}{\sqrt{\left\lfloor S\right\rfloor }}\sum_{s=1}^{\left\lfloor S\right\rfloor }Y_{s}+\frac{1}{\sqrt{S}}\int_{\left\lfloor S\right\rfloor }^{S}h_{t}\left(X\left(t\right)\right)dM\left(t\right).
\]
Clearly, $\left\lfloor S\right\rfloor /S\rightarrow1$. Moreover,
by arguments similar to the ones used to verify the third condition
(iii.) above, deduce that the last term on the r.h.s. is $o_{p}\left(1\right)$.
This shows the result using $\sigma_{S}$ as scaling sequence rather
than $\hat{\sigma}_{S}$. However, $\left|\hat{\sigma}_{S}^{2}-\sigma_{S}^{2}\right|=\left|\frac{1}{S}\int_{0}^{S}h_{t}^{2}\left(X\left(t\right)\right)dM\left(t\right)\right|\rightarrow0$
a.s., and we can use $\hat{\sigma}_{S}^{2}$ to define the t-statistic.
This completes the proof.

\section{Details Regarding Section \ref{Section_SimulationHawkesManyCovariates}
\label{Section_additionalHawkesSimulationDetails}}

Define $Y_{i}:=\exp\left\{ g_{0}\left(X\left(T_{i}\right)\right)\right\} $
and $Z_{i}:=\sum_{T_{j}\leq T_{i}}e^{-a_{0}\left(T_{i}-T_{j}\right)}$,
and recall $R\left(T_{i+1}\right)=T_{i+1}-T_{i}$. Note that for $t\in(T_{i},T_{i+1}]$,
$\lambda\left(t\right)=\left(c_{0}+Z_{i}e^{-a_{0}\left(t-T_{i}\right)}\right)Y_{i}$.
In consequence, 
\[
\Lambda\left((T_{i},T_{i+1}]\right)=\int_{T_{i}}^{T_{i+1}}\lambda\left(t\right)dt=\left[c_{0}R\left(T_{i+1}\right)+\frac{Z_{i}}{a_{0}}\left(1-e^{-a_{0}R\left(T_{i+1}\right)}\right)\right]Y_{i}
\]
is mean one, exponentially distributed, conditioning on $\mathcal{F}_{i}:=\left(T_{i},Z_{i},Y_{i}\right)$.
Moreover, $Z_{i}=Z_{i-1}e^{-a_{0}\left(T_{i}-T_{i-1}\right)}+1$ with
$Z_{0}=1$. Hence, define $c_{1}=c_{0}Y_{i}$, $c_{2}=Y_{i}Z_{i}$,
and simulate i.i.d. $\left[0,1\right]$ uniform random variables $U_{i}$'s.
We simulate $R\left(T_{i}\right)$ setting it equal to the $s$ that
solves $c_{1}s+\frac{c_{2}}{a_{0}}\left(1-e^{-a_{0}s}\right)=-\ln U_{i}$.

Given an initial guess $\left(2,1.5\right)$ of $\left(c_{0},a_{0}\right)=\left(2,1.3\right)$
we estimate $\exp\left\{ g_{T}\left(X\left(t\right)\right)\right\} $.
Given $\exp\left\{ g_{T}\left(X\left(t\right)\right)\right\} $ we
estimate $c$ and $a$ in $\left(c+\sum_{T_{i}<t}e^{-a\left(t-T_{i}\right)}\right)\exp\left\{ g_{T}\left(X\left(t\right)\right)\right\} $.
We perform a second iteration. 

Estimation of $g$ is done using the algorithm in Section \ref{Section_EstimationDetails}.
In this case, the relevant part of the likelihood is 
\[
\sum_{i=1}^{n}g\left(T_{i-1}\right)-\sum_{i=1}^{n}\exp\left\{ g\left(T_{i-1}\right)\right\} \Delta_{i}
\]
where 
\[
\Delta_{i}=cR\left(T_{i}\right)+\frac{Z_{i-1}}{a}\left(1-e^{-aR\left(T_{i}\right)}\right)
\]
and $c$ and $a$ are set to their guess/estimated values. Estimation
of $c$ and $a$ is via maximum likelihood given $\exp\left\{ g_{T}\left(X\left(t\right)\right)\right\} $. 
\end{document}